\documentclass[12pt]{article}
\usepackage[utf8]{inputenc}
\RequirePackage[OT1]{fontenc}
\RequirePackage[numbers]{natbib}

\usepackage{graphicx,amsmath,amsfonts,amssymb,amsthm,psfrag, framed,thmtools,cancel,float}
\usepackage[dvipsnames]{xcolor}
\usepackage{fullpage}
\declaretheoremstyle[
spaceabove=6pt, spacebelow=6pt,
headfont=\normalfont\bfseries,
notefont=\mdseries, notebraces={(}{)},
headpunct=.\,— ,
bodyfont=\normalfont,
numbered=no
]{solu}

\declaretheorem[name=Calculus, style=solu, preheadhook = \color{Green}]{Calculus}

\usepackage{comment}
\excludecomment{Com}
\excludecomment{Calculus}
\excludecomment{Mistake}

\newcommand{\ds}{\displaystyle}

\newcommand{\dint}{\mathrm{d}}

\newcommand{\ai}{A^{(i)}}
\newcommand{\af}{A^{(f)}}

\newtheorem{thm}{Theorem}[section]
\newtheorem{corollary}[thm]{Corollary}
\newtheorem{lemma}[thm]{Lemma}

\newtheorem{definition}[thm]{Definition}

\newtheorem{rem}[thm]{Remark}

\numberwithin{equation}{section}

\title{Strong approximation of Bessel processes\protect\thanks{This work has been supported by the project PERISTOCH ANR–19–CE40–0023, 2020--2024 of the French National Research Agency (ANR)}}

\begin{document}
\author{Madalina Deaconu$^1$ and Samuel Herrmann$^2$
\\[5pt]
\small {$^1$Universit\'e de Lorraine, CNRS, Inria, IECL, F-54000 Nancy, France,}\\
\small{Madalina.Deaconu@inria.fr}\\[5pt]
\small{$^2$Institut de Math{\'e}matiques de Bourgogne (IMB) - UMR 5584, CNRS,}\\
\small{Universit{\'e} de Bourgogne Franche-Comt\'e, F-21000 Dijon, France} \\
\small{Samuel.Herrmann@u-bourgogne.fr}
}

\maketitle

\begin{abstract}

 We consider the path approximation of Bessel processes and develop a new and efficient algorithm. This study is based on a recent work by the authors, on the path approximation of the Brownian motion, and on the construction of specific own techniques. It is part of the family of the so-called $\varepsilon$-strong approximations. More precisely, our approach constructs jointly the sequences of exit times and corresponding exit positions of some well-chosen domains, the construction of these domains being an important step. Based on this procedure, we emphasize an algorithm which is easy to implement. Moreover, we can develop the method for any dimension.  We treat separately the integer dimension case and the non integer framework, each situation requiring appropriate techniques. In particular, for both situations, we show the convergence of the scheme and provide the control of the efficiency with respect to the small parameter  $\varepsilon$. We expand the theoretical part by a series of numerical developments. 
 
\end{abstract}
{\small
\noindent \textbf{Key words:} Strong approximation, path simulation, Bessel process, Brownian exit time.
\par\medskip

\noindent \textbf{2010 AMS subject classifications:} primary 
65C05;   	
secondary 
60J60,       
60J25,  	
60G17,      
60G50.  	
} 	

\section*{Introduction}\label{sec:intro}

Diffusion processes play a central role in the modelling and study of the behaviour of physical phenomena, of biological problems or of financial products pricing, it is thus of prime interest to develop numerical approaches to characterize and analyze their stochastic trajectories. However, a trajectory is an infinite mathematical object which cannot be generated directly, an approximation procedure and its corresponding error control need therefore to be emphasized. 
 
The Euler scheme is one of the classical standard schemes to get numerical approximated solutions of stochastic differential equations. Indeed, a first method to approximate stochastic processes is the common time-discretization procedure: only the values of the diffusion process on some finite deterministic time grid $t_1<t_2<\ldots<t_n$ are described (or approximated), as in the usual Euler or modified Euler scheme and the literature contains many convergence results in the small time step limit.
In the standard case, that is under the conditions that ensure the existence and uniqueness of the solution of the SDEs, the numerical analysis is well-developed and a large variety of different numerical approximation schemes is available. We refer, in this framework of time splitting procedure to the important work  \cite{kloedenplaten1999Euler}. 
A huge literature is studying this approach and we can find results on the weak convergence  \cite{bally1995weakeuler} or the strong convergence. It is thus well-known, that under suitable conditions on the coefficients of the SDEs, the Euler scheme has strong rate of convergence $1/2$, \cite{kloedenplaten1999Euler}. A very good review on the results on the Euler method and its higher-order extensions can be found in \cite{jourdainhiga2011euler}, and for particular  diffusions interesting techniques are developed for instance in \cite{dereich2012eulercir}, \cite{gobet2000}, \cite{muller2020eulernonsmooth}, and many others. 

Such schemes often suffer in terms of efficiency as the computational time is high and  it is difficult to solve the trade-off between reducing the error and obtaining a satisfactory computational time. Thus, in non standard cases, other methods need to be developed.

An alternative approach for one-dimensional diffusions is to squeeze the stochastic trajectory $(X_t)_{t\ge 0}$ under observation inbetween two simple to exhibit paths, depending on a small parameter $\varepsilon>0$: an upper and a lower trajectory $(X_t^{\uparrow,\varepsilon})_{t\ge 0}$ respectively  $(X_t^{\downarrow,\varepsilon})_{t\ge 0}$.  Obviously, these two curves need first to be easy to generate numerically (we should avoid infinite mathematical objects) and secondly be such that their difference can be controlled with respect to the parameter $\varepsilon$: on any finite time interval $[0,T]$,
\[
X_t^{\uparrow,\varepsilon}\ge X_t\ge X_t^{\downarrow,\varepsilon},\quad \forall t\in [0,T]\quad \mbox{and}\quad \lim_{\varepsilon\to 0}\sup_{t\in [0,T] }(X_t^{\uparrow,\varepsilon}-X_t^{\downarrow,\varepsilon})=0\quad \mbox{a.s.}
\]
There is a challenging dual objective: to point out some upper and lower convergent bounding processes $X^{\uparrow,\varepsilon}$ and $X^{\downarrow,\varepsilon}$ on one hand, and to get a precise convergence result on the other hand. One interesting approach is to link the construction of the bounds with the simulation of the diffusion exiting from thin horizontal layers. Such an approach can be seen through to its successful completion using both the precise description of the Brownian paths, in particular Brownian meanders, and the exact simulation method (rejection sampling), see the seminal paper of Chen and Huang \cite{chen2013localization} and subsequent developments concerning killed diffusions \cite{casella2008exact}, jump diffusions \cite{giesecke2013exact} or further techniques linked to the $\varepsilon$-strong approximation \cite{pollock2016exact}. 
These approaches concern mainly classical diffusions or jump diffusions with regular coefficients. 

\begin{figure}[h]
\centering
\includegraphics[width=8.5cm]{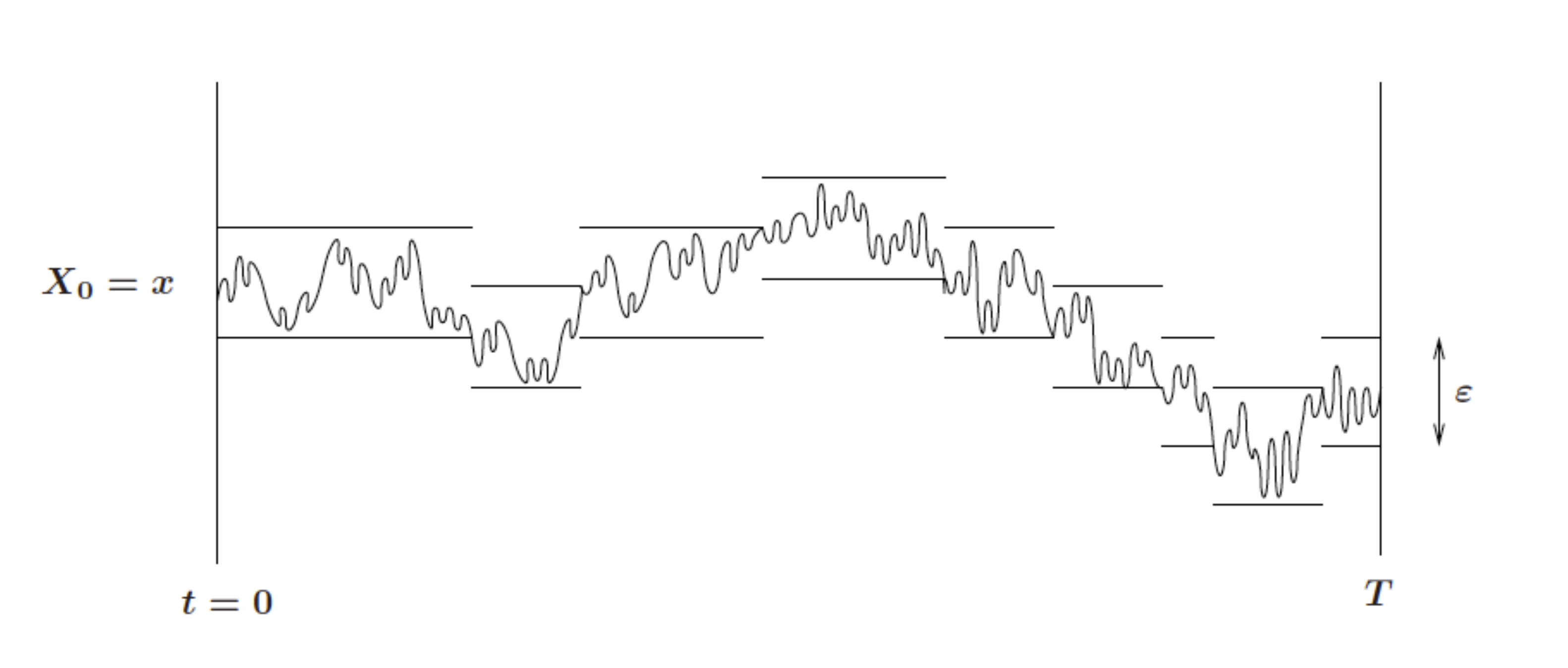}\hspace*{-0.5cm}\includegraphics[width=8.5cm]{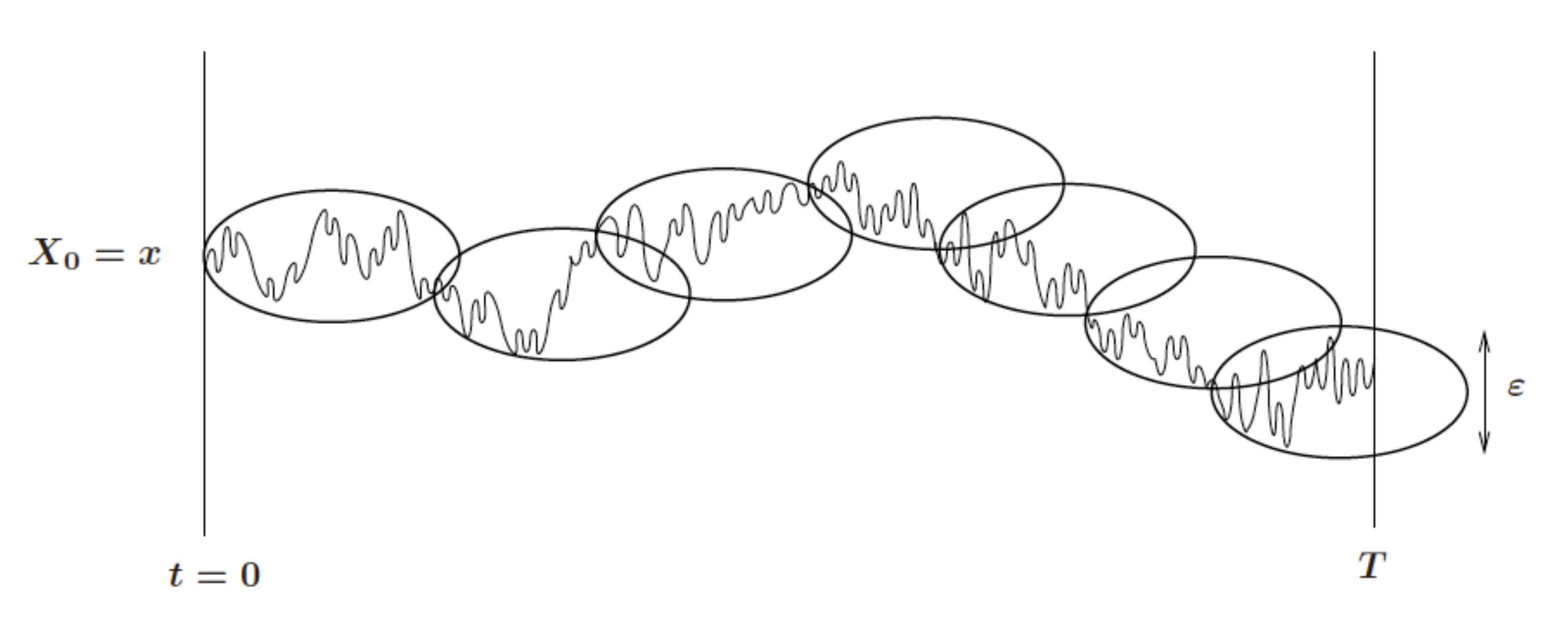}
\caption{\small {Illustration of $\varepsilon$-strong approximation of diffusion paths based on successive exit times of small intervals (left)
or successive exits of small spheroids (right)}}
\label{fig:my_label0}
\end{figure}

The aim of this paper is to develop a different construction for the $\varepsilon$-strong approximation
of Bessel trajectories. The stochastic differential equation satisfied by the Bessel process
presents a singular drift coefficient. We shall thus emphasize an alternative approach which
is not based on exit times of thin horizontal layers but rather on exit times of specific
spheroids (see Figure \ref{fig:my_label0}). There is a substantial numerical gain to adopt this new point of
view since the spheroid exit times are easy to generate and the rejection sampling linked to
the exact simulation of a paths skeleton can be avoided. The new algorithm we propose concerns Bessel
processes of integer or non integer dimensions and is based on observations of the Brownian
trajectories.
Let us first introduce the definition of the $\varepsilon$-strong approximation that we shall consider
throughout our study.
\begin{definition}\label{def}
The random process $(y_t^\varepsilon)$ is an $\varepsilon$-strong approximation of the diffusion process $(X_t)$ if there exists $(x^\varepsilon_t)$ satisfying
\begin{equation}
\label{eq:def}
\sup_{t\in[0,T]}|X_t-x^\varepsilon_t|\le \varepsilon\quad\mbox{a.s.}
\end{equation}
such that $(y^\varepsilon_t)$ and $(x^\varepsilon_t)$ are identically distributed.
\end{definition}
In the Brownian framework, such an approximation is available using the exit times
of specific spheroids \cite{deaconu2020strong}. Let us consider the curve directly linked to the shape of the
$\delta$-dimensional heat ball:
\begin{equation}
\label{def:phi}
\phi_{\delta,\varepsilon}(t):=\sqrt{\delta t\ln\Big(\frac{e\varepsilon^2}{\delta t }\Big)}, \quad \mbox{for}\ t\in I_{\delta, \varepsilon}:=[0,r_{\delta,\varepsilon}]\ \mbox{with}\ r_{\delta,\varepsilon}=\frac{e\varepsilon^2}{\delta}.
\end{equation}
We can notice that the maximum of the curve is equal to $\varepsilon$ and is reached for $t_{\rm max}=\varepsilon^2/\delta$. The approximation is based on the so-called Brownian skeleton $({\rm BS})_\eta$. We recall now the particular situation $\eta\equiv 1$.

\begin{framed}
\centerline{\sc Brownian Skeleton $({\rm BS})_1$}

\vspace*{0.2cm}
\begin{enumerate}
\item Let $(A_n)_{n\ge 1}$  be i.i.d. random variables with gamma distribution ${\rm Gamma}(3/2,2)$.
\item Let $(Z_n)_{n\geq 1}$ be i.i.d. Rademacher random variables (taking values +1 or -1 with probability 1/2). The sequences  $(A_n)_{n\ge 1}$ and $(Z_n)_{n\geq 1}$ are  independent.
\end{enumerate}

\noindent
{\bf  Definition:} Let $\varepsilon>0$. The \emph{Brownian skeleton} $({\rm BS})_1$ is defined by
\[
\Big((u_n^\varepsilon,s_n^\varepsilon)_{n\ge 1},(x_n^\varepsilon)_{n\ge 0}\Big)\quad\mbox{with}\quad \left\{
\begin{array}{l}
u_n^\varepsilon=\varepsilon^2 \,e^{1-A_n},\quad 
 s^\varepsilon_n=\ds\sum_{k=1}^nu_k^\varepsilon,\\[18pt]
  x_{n}^\varepsilon=x_{n-1}^\varepsilon+ Z_n \,\phi_{1,\varepsilon}( u_{n}^\varepsilon),\ \forall n\ge 1 \\
 \end{array}
 \right.
\]
and $x_0^\varepsilon=x$.
\end{framed}

The authors proved in \cite{deaconu2020strong} that the piecewise constant function 
\[
x^\varepsilon_t=\sum_{n\ge 0}x_n^\varepsilon1_{\{  s_n^\varepsilon \le t<s^\varepsilon_{n+1}\}}
\]
is an $\varepsilon$-strong approximation of the Brownian paths starting in $x$. Moreover the number of random points used to approximate the trajectory on a fixed time interval $[0,T]$, denoted by $N_T^\varepsilon$
\begin{equation}
\label{def:N}
N_T^\varepsilon:=\inf\{n\ge 1:\ s_n^\varepsilon\ge T\},
\end{equation}
satisfies
\[
\lim_{\varepsilon\to 0}\varepsilon^2\mathbb{E}[N^\varepsilon_T]=\kappa \cdot T\quad\mbox{with}\quad \kappa=3^{3/2}e^{-1}.
\]
This means that, the cost of the numerical approximation is of the order $T/\varepsilon^2$.

 The approximation procedure presented in \cite{deaconu2020strong} is based on a piecewise constant function whose intersection with the Brownian paths corresponds to the set of points $\{(s_n^\varepsilon,x_n^\varepsilon):\ n\ge 0\}$. The sequence $(s_n^\varepsilon)_{n\ge 0}$ is built using the random variables $(u_n^\varepsilon)_{n\ge 1}$ which represent Brownian exit times from some typical spheroid defined by $\phi_{1,\varepsilon}$. Thus, the sequence of points $\{(s_n^\varepsilon,x_n^\varepsilon):\ n\ge 0\}$ is obtained as the successive Brownian exit points of time-space spheroids of size $\varepsilon$: the Brownian path is therefore completely controlled inbetween two successive points of the skeleton.

In this paper we aim to adapt and develop this technique to the approximation of Bessel processes. Whenever the dimension of the Bessel process is an integer, the paths are distributed like the norm of a multidimensional Brownian motion. Consequently, the successive exit points of spheroids for that Brownian trajectory allows to build a Bessel skeleton, the main ingredient of the $\varepsilon$-strong approximation. In the framework of Bessel processes with non integer dimension, the construction of the algorithm is more difficult: we still use successive exit points of spheroids but we first have to decompose the Bessel paths into two independent parts following the flagship identity of Shiga and Watanabe \cite{Shiga}.

In both cases, for integer or non integer dimensions, we develop the corresponding approximation scheme and prove the results that characterize and control its behaviour.

\section{Bessel processes with integer dimension}\label{sec:integer}
The Bessel process of dimension $\delta$ is the unique solution of the following stochastic differential equation:
\begin{equation}
\label{eq:def:Be}
Z_t^{\delta,y}=y+\frac{\delta-1}{2}\int_{0}^t (Z_s^{\delta,y})^{-1}\,ds+B_t,\quad t\ge 0,
\end{equation}
where $y\ge 0$ is the deterministic initial value of the process and $(B_t)_{t\ge 0}$ stands for a one dimensional standard Brownian motion. In particular, for integer values of $\delta$, the Bessel process and the norm of a $\delta$-dimensional Brownian motion are identically distributed. This classical property shall play a crucial role in the sequel. Let us denote by $\nu$ the so-called Bessel index related to the dimension by the following relation $\nu=\frac{\delta}{2}-1$.
\begin{framed}
\centerline{\sc Bessel Skeleton $({\rm BeS})_\delta$ -- integer dimension}

\vspace*{0.2cm}
\begin{enumerate}
\item Let $(A_n)_{n\ge 1}$  be a sequence of i.i.d. random variables with gamma distribution ${\rm Gamma}(\nu+2,1/(\nu+1))$ that is the shape equals $\nu+2$ and the scale $1/(\nu+1)$.
\item Let $(V_n)_{n\ge 1}$ be a sequence of i.i.d. uniformly distributed random vectors on the boundary of the unitary sphere $\mathcal{S}^\delta$. We denote by $\pi_1(V_n)$ the projection on
the first coordinate. The sequences 
 $(A_n)_{n\ge 1}$ and $(V_n)_{n\ge 1}$ are assumed to be independent. 
\end{enumerate}

\noindent
{\bf  Definition:} For $\varepsilon>0$, the \emph{Bessel skeleton} $({\rm BeS})_\delta$ is given by $\Big((u_n^\varepsilon,s_n^\varepsilon)_{n\ge 1},(y_n^\varepsilon)_{n\ge 0}\Big)$
\[
\quad\mbox{with}\  \left\{
\begin{array}{l}
u_n^\varepsilon=\frac{\varepsilon^2}{\delta} \,e^{1-A_n},\quad 
 s^\varepsilon_n=\ds\sum_{k=1}^nu_k^\varepsilon,\\[18pt]
  y_n^\varepsilon=\Big((y_{n-1}^\varepsilon)^2+ 2\pi_1(V_n)\,y_{n-1}^\varepsilon \,\phi_{\delta,\varepsilon}( u_{n}^\varepsilon)+\phi_{\delta,\varepsilon}^2( u_{n}^\varepsilon)\Big)^{1/2}, \forall n\ge 1 \\
 \end{array}
 \right.
\]
and $y_0^\varepsilon=y$.

\end{framed}
The Bessel skeleton permits to construct an approximation of the Bessel trajectory. The main idea leading to this construction is first to relate the Bessel process to the norm of a $\delta$-dimensional Brownian motion. Secondly we replace the Brownian trajectory by a Brownian skeleton: a random walk corresponding to the successive exits of a sequence of small time-space spheroids.
\begin{thm}\label{thm:Bess1}
Let $\varepsilon>0$ and let us consider a Bessel skeleton $({\rm BeS})_{\delta}$ with $\delta\in\mathbb{N}^*$. Then $y^\varepsilon_t=\sum_{n\ge 0}y_n^\varepsilon1_{\{  s_n^\varepsilon \le t<s^\varepsilon_{n+1}\}}$ is an $\varepsilon$-strong approximation of the Bessel paths starting in $y$, solution of \eqref{eq:def:Be}. 
The number of approximation points $N_T^\varepsilon$ on the fixed interval $[0,T]$ satisfies:
\begin{equation}\label{eq:thm:1}
\lim_{\varepsilon\to 0}\varepsilon^2\,\mathbb{E}[N^\varepsilon_T] =\frac{\delta T}{e}\Big(\frac{\nu+2}{\nu+1}\Big)^{\nu+2}. 
\end{equation}
Moreover the following CLT is observed:
\[
\ds\lim_{\varepsilon \to 0}\ \frac{\sqrt{e}}{\varepsilon \sigma\sqrt{\delta T}}\ \Big(\frac{\nu+1}{\nu+2}\Big)^{3\nu/2+3}\Big[ \varepsilon^2 N^\varepsilon_T -\frac{\delta T}{e}\Big( \frac{\nu+2}{\nu+1} \Big)^{\nu+2}\Big] = G \quad \mbox{in distribution,}
\]
with $G$ a $\mathcal{N}(0,1)$ standard Gaussian random variable and 
\begin{equation}
\label{eq:thm:11}
\sigma^2=\Big( \frac{\nu+1}{\nu+3} \Big)^{\nu+2}-\Big(\frac{\nu+1}{\nu+2}\Big)^{2\nu+4}.
\end{equation}
\end{thm}
It is important to notice that Theorem \ref{thm:Bess1} leads to confidence intervals for the number of approximated points which represents the efficiency of the approximation algorithm. 
\begin{proof} We construct the proof in several steps.\\[5pt]
\emph{Step 1: $\varepsilon$-strong approximation.}\\[5pt]
Let us consider the $\delta$-dimensional Bessel process $(Z^{\delta,y}_t,\, t\ge 0)$ starting in $y$. We introduce the vector $\hat{y}=(y,0,\ldots,0)\in\mathbb{R}^\delta$. It is well-known that $(Z^{\delta,y}_t,\, t\ge 0)$ has the same distribution as $(\Vert y+W_t \Vert,\, t\ge 0)$ where $W$ stands for a standard $\delta$-dimensional Brownian motion. It suffices therefore to strongly approximate the Brownian norm since the strong approximation is based on an identity in law.\\
Let us now build a sequence of points $(t_n,z_n)_{n\ge 0}$ belonging to the trajectory of the $\delta$-dimensional Brownian motion $(t,\hat{y}+W_t)_{t\ge 0}$ and satisfying $t_{n+1}\ge t_n$ for any $n\ge 0$. The sequence of times is defined by 
\begin{equation}
\label{eq:def:times}
t_{n+1}:=\inf\{t> t_{n}:\, \Vert W_t -W_{t_n}\Vert =\phi_{\delta,\varepsilon}(t-t_n) \}\quad\mbox{and}\quad t_0=0.
\end{equation}
These times represent the successive exit times of spheroids sequence (also called heat balls) whose boundary shape corresponds to the function $\phi_{\delta,\varepsilon}$ defined by \eqref{def:phi}. In order to observe points belonging to the path, we set $z_n:=\hat{y}+W_{t_n}$. Due to the definition of the stopping times and since the maximum of the function $\phi_{\delta,\varepsilon}$ equals $\varepsilon$, we get
\[
\Vert \hat{y}+W_t-z_n\Vert \le \varepsilon,\quad \forall t\in [ t_n, t_{n+1}],\quad \forall n\ge 0.
\]
This means that, if we denote by $(z_t)_{t\ge 0}$ the random function satisfying
\[
z_t:=\sum_{n\ge 0}z_n1_{\{ t_n\le t<t_{n+1} \}},\quad t\ge 0,
\]
then we have that $\Vert \hat{y}+W_t-z_t\Vert \le \varepsilon$, for all $t\ge 0$, almost surely. Moreover $\Vert \hat{y}+W_t-z_t\Vert=0$ as soon as $t=t_n$ and $n\ge 0$. This approximation of the $\delta$-dimensional Brownian trajectory obviously allows to approximate its Euclidean norm. The second triangle inequality leads to
\[
|\Vert \hat{y}+W_t\Vert-\Vert z_t\Vert|\le \Vert \hat{y}+W_t-z_t\Vert \le \varepsilon,\quad \forall t\ge 0.
\]
Since the strong approximation is based on an identity in distribution, $(\Vert z_t \Vert)_{t\ge 0}$ is an $\varepsilon$-strong approximation of the Bessel path.\\[5pt] 
\noindent\emph{Step 2: Relation to the Bessel skeleton ${\rm (BeS)}_{\delta}$.}\\[5pt]
To construct a typical approximated trajectory, it suffices to generate the sequence of successive times $(t_n)_{n\ge 0}$ and the associated sequence $(z_n)_{n\ge 0}=(W_{t_n})_{n\ge 0}$. It corresponds in fact to the sequences of exit times and exit locations of spheroids. In \cite{Deaconu-Herrmann-2013}, the authors described the distribution of these two sequences. We note that
\begin{itemize}
\item the random variables $(t_{n+1}-t_n)_{n\ge 0}$ are independent and identically distributed. Moreover $\delta (t_{n+1}-t_n)/(e\varepsilon^2)$ has the same distribution as $e^{-A}$ where $A\sim {\rm Gamma}(\nu+2,1/(\nu+1))$.
\item the $\delta$-dimensional Brownian motion satisfies the rotational invariance property. Therefore $z_0=\hat{y}$ and $z_{n+1}$ is uniformly distributed on the sphere of center $z_n$ and radius $\phi_{\delta,\varepsilon}(t_{n+1}-t_n)$. Consequently 
\[
\Vert z_{n+1}\Vert^2=\Vert z_{n}\Vert^2+2\pi_1(V)z_n\phi_{\delta,\varepsilon}(t_{n+1}-t_n)+\phi_{\delta,\varepsilon}^2(t_{n+1}-t_n),
\]
where $V$ is uniformly distributed on the unitary sphere of dimension $\delta$ and $\pi_1$ stands for the projection on the first coordinate.
\end{itemize}
We deduce that $(t_n,\Vert z_n\Vert)_{n\ge 0}$ and the Bessel skeleton (BeS) $(s_n^\varepsilon,y_n^\varepsilon)_{n\ge 0}$  are identically distributed and consequently, the process $(y^\varepsilon_t)_{t\ge 0}$ defined in the statement of Theorem \ref{thm:Bess1} defines an $\varepsilon$-strong approximation of the Bessel process. \\[5pt]
\noindent\emph{Step 3: Number of points needed to cover $[0,T]$.}\\[5pt]
Let us now focus our attention on the number of spheroids used until a fixed time $T$, defined by $N^\varepsilon_T:=\inf\{n\ge 0:\ t_n\ge T\}$. The arguments used here are similar to those developed in \cite{deaconu2020strong} (Proposition 2.2). We denote by $(\hat{N}_t)_{t\ge 0}$ the Poisson process with independent and identically distributed arrivals $(e^{1-A_n})_{n\ge 1}$, defined by the Bessel skeleton. Then the classical asymptotic result holds
\begin{equation}\label{eq:limithm}
\lim_{t\to\infty}\frac{\mathbb{E}[\hat{N}_t]}{t}=\frac{1}{\mathbb{E}[e^{1-A_1}]}=\frac{1}{e\mathcal{L}_{A_1}(1)}=e^{-1}\Big(\frac{\nu+2}{\nu+1}\Big)^{\nu+2}.
\end{equation}
Here $\mathcal{L}_{A_1}(s)$ stands for the Laplace transform of the variate $A_1\sim {\rm Gamma}(\nu+2,1/(\nu+1))$, that is
\[
\mathcal{L}_{A_1}(s)=\Big(\frac{s}{\nu+1}+1\Big)^{-\nu-2},\quad \forall s\ge 0.
\]
Furthermore, the central limit theorem holds: if we denote by $\mu=\mathbb{E}[e^{1-A_1}]= e\Big(\frac{\nu+1}{\nu+2}\Big)^{\nu+2}$ and use the parameter $\sigma$ defined in \eqref{eq:thm:11}, then \({\rm Var}(e^{1-A_1})=e^2\sigma^2\) and 
\begin{equation}\label{eq:CLT}
\lim_{t\to\infty}\sqrt{\frac{t\mu^3}{e^2\sigma^2}}\Big(\frac{\hat{N}_t}{t}-\frac{1}{\mu}\Big)=G,\quad\mbox{in distribution,}
\end{equation}
where $G$ is a $\mathcal{N}(0,1)$ standard Gaussian random variable. These two asymptotic results described in \eqref{eq:limithm} and \eqref{eq:CLT} are related to the behaviour of $N_T^\varepsilon$ in the $\varepsilon$ small limit since   
\[
N_T^\varepsilon=\hat{N}_{\frac{T\delta}{\varepsilon^2}},\quad \forall \varepsilon>0.
\]
The announced statement is therefore an easy consequence of the previous identity.
\end{proof}
\section{Bessel processes with non integer dimension}\label{sec:noninteger}
In the previous section, it was crucial that the dimension $\delta$ of the Bessel process was an integer: this allows in particular to associate the Bessel paths with the norm of the $\delta$-dimensional Brownian motion. In the general case the dimension of the Bessel process defined in \eqref{eq:def:Be} is just a real valued parameter so we also need to develop an $\varepsilon$-strong approximation procedure for noninteger dimensions. In the particular case: $\delta\in[1,\infty[\setminus\mathbb{N}$, the crucial tool is the argument developed by Shiga and Watanabe \cite{Shiga} and already used for simulation purposes in \cite{Deaconu-Herrmann-2017}. The Bessel process of dimension $\delta$ starting in $y\ge 0$ has the same distribution as the sum of two \emph{independent} processes:
\begin{itemize}
\item a Bessel process of dimension $\delta_i:=\lfloor \delta \rfloor$ (integer dimension) starting in $y$ (the corresponding index is denoted $\nu_i$),
\item a Bessel process of dimension $\delta_f:=\delta-\lfloor \delta \rfloor$ (fractional dimension) starting in $0$ (the corresponding index is denoted $\nu_f$).
\end{itemize}
A wise combination of the construction developed in Section \ref{sec:integer} on one hand, and the identity of Shiga-Watanabe on the other hand, allows to develop an adapted procedure in the general framework. 
\subsection*{A rejection sampling algorithm}
Before defining the general Bessel skeleton, we need to introduce the generation of a particular family of random variables already mentioned in \cite{Deaconu-Herrmann-2017}. The probability distribution under consideration is deeply related to the Bessel process of dimension $\delta$ exiting from a spheroid of size $\varepsilon$. Nevertheless, we prefer to use in this paragraph generic constants $\alpha>0$ and $\beta>0$ for notational simplicity. In the sequel we are going to fix $\alpha=\delta/2$ and $\beta=2e\varepsilon^2/\delta$. Let us introduce the function $u$ defined by
\begin{equation}
\label{eq:defdeu}
u_{\alpha,\beta}(t,x)= \frac{1}{t^\alpha}\,\exp\Big\{-\frac{x^2}{t}\Big\}-\frac{1}{\beta^\alpha} 
,\quad \forall(t,x)\in\mathbb{R}_+^*\times \mathbb{R}_+^*,
\end{equation}
and the associated normalization constant $\kappa_{\alpha,\beta}$:
\begin{equation}
\label{eq:defconst}
\kappa_{\alpha,\beta,t}:=\int_0^{\rho_{\alpha,\beta,t}}u_{\alpha,\beta}(t,y)\,y^{2\alpha-1}\,\dint y\quad\mbox{with}\ \rho_{\alpha,\beta,t}:=\sqrt{\alpha t \ln\Big(\frac{\beta}{t}\Big)}.
\end{equation}
The constant $\rho_{\alpha,\beta,t}$ corresponds to the positive zero of the function $x\mapsto u_{\alpha,\beta}(t,x)$. We deduce therefore that 
\begin{equation}
x\mapsto \kappa^{-1}_{\alpha,\beta,t}\,u_{\alpha,\beta}(t,x)x^{2\alpha-1}\,1_{[0,\rho_{\alpha,\beta,t}]}(x)
\label{eq:dens}
\end{equation}
is a probability distribution function. A random variable whose density is given by \eqref{eq:dens} can be generated using the following rejection sampling. 
\begin{framed}
\centerline{\sc Conditional distribution $({\rm CD})_{\alpha,\beta}^t$}

\vspace*{0.2cm}
\begin{enumerate}
\item Let $(R_n)_{n\ge 1}$  be a sequence of uniformly distributed i.i.d. random variables on the interval $[0,1]$. 
\item Let $(V_n)_{n\ge 1}$ be another sequence of i.i.d. uniformly distributed random variables on $[0,1]$. The sequences 
 $(R_n)_{n\ge 1}$ and $(V_n)_{n\ge 1}$ are assumed to be independent. 
\end{enumerate}
{\bf Initialization:} $n=1$.\\[5pt]
{\bf While} \( u_{\alpha,\beta}(t,0)\, R_n> u_{\alpha,\beta}(t,\rho_{\alpha,\beta,t}\,V_n^{1/(2\alpha)})\) {\bf set} $n\leftarrow n+1$;\\[5pt]
{\bf Outcome:} $\mathcal{Z}=\rho_{\alpha,\beta,t}\,V_n^{1/(2\alpha)}$.
\end{framed}
This algorithm is of prime importance in the study of Bessel processes. Indeed let us consider a Bessel process $(Z^{\delta,0}_t)_{t\ge 0}$ starting in $0$ and with dimension $\delta>0$ and let us denote $\tau_\phi$ the first passage time through the curved boundary given by \eqref{def:phi}. We omit the dependence with respect to the parameters $\delta$ and $\varepsilon$ for notational simplicity. The following identity in distribution holds.
\begin{lemma} Let $0<t<r_{\delta,\varepsilon}=e\varepsilon^2/\delta$. The outcome of Algorithm $({\rm CD})_{\alpha,\beta}^{2t}$, with $\alpha=\delta/2$ and $\beta=2e\varepsilon^2/\delta$, has the same distribution as the conditional distribution of $Z_t^{\delta,0}$ given $\tau_\phi>t$.
\label{lem:identlaw}
\end{lemma}
\begin{proof} Since Algorithm $({\rm CD})_{\alpha,\beta}^t$ is an acceptance-rejection sampling, we can easily describe the distribution of its outcome. Let $\psi$ be any non negative measurable function. We consider $R$ and $V$ two independent uniformly distributed r.v., then
\begin{align}
\mathbb{E}[\psi(\mathcal{Z})]=\frac{\nu(\psi)}{\nu(1)}\quad\mbox{where}\
\nu(\psi):=\mathbb{E}\Big[\psi(\rho_{\alpha,\beta,t}\,V^{1/(2\alpha)}) 1_{\{ u_{\alpha,\beta}(t,0)\, R\le  u_{\alpha,\beta}(t,\rho_{\alpha,\beta,t}\,V^{1/(2\alpha)}) \}}\Big].
\label{eq:reject}
\end{align}
Using the change of variables $y=\rho_{\alpha,\beta,t}\,x^{1/(2\alpha)}$ permits to obtain
\begin{align}
\nu(\psi)&=\int_0^1\psi(\rho_{\alpha,\beta,t}\,x^{1/(2\alpha)})\,\frac{u_{\alpha,\beta}(t,\rho_{\alpha,\beta,t}\,x^{1/(2\alpha)})}{u_{\alpha,\beta}(t,0)}\,\dint x\nonumber\\
&=\frac{2\alpha}{\rho_{\alpha,\beta,t}^{2\alpha}}\int_0^{\rho_{\alpha,\beta,t}}\psi(y)\,\frac{u_{\alpha,\beta}(t,y)}{u_{\alpha,\beta}(t,0)}\,y^{2\alpha-1}\dint y.\label{eq:dev}
\end{align}
Combining \eqref{eq:reject} and \eqref{eq:dev} proves that the p.d.f. of the random variable $\mathcal{Z}\sim({\rm CD})_{\alpha,\beta}^t$ corresponds to the function introduced in \eqref{eq:dens}. After setting $\alpha=\delta/2$ and $\beta=2e\varepsilon^2/\delta$, we deduce that the density of  $\mathcal{Z}\sim({\rm CD})_{\alpha,\beta}^{2t}$ corresponds to the function
\[
x\mapsto \left(\frac{2}{(2t)^{\delta/2}\Gamma(\delta/2)}\exp\Big\{ -\frac{x^2}{2t} \Big\}-\frac{2}{\Gamma(\delta/2)}\Big(\frac{\delta}{2e\varepsilon^2}\Big)^{\delta/2}\right)\, x^{\delta-1}1_{\{0<x<\phi_{\delta,\varepsilon}(t)\}}
\]
which is exactly the conditional density of $Z_t^{\delta,0}$, given $\tau_\phi>t$ (see, for instance, \cite{Deaconu-Herrmann-2017}).
\end{proof}
\begin{rem}
\label{rem:efficiency}
The algorithm $({\rm CD})^t_{\alpha,\beta}$ is based on a rejection sampling method, it is therefore straightforward to describe the efficiency of the procedure. It is well known that the number of trials corresponds to a geometrically distributed random variable denoted by $N$, with parameter $\nu(1)$, $\nu$ being defined in \eqref{eq:reject}. We deduce from \eqref{eq:dev} that
\begin{align*}
\mathbb{E}[N]=\frac{1}{\nu(1)}=\frac{\rho_{\alpha,\beta,t}^{2\alpha}}{2\alpha}\frac{u_{\alpha,\beta}(t,0)}{\kappa_{\alpha,\beta,t}}.
\end{align*}
An integration by parts allows to compute the value of the constant $\kappa_{\alpha,\beta,t}$, by introducing the incomplete Gamma function:
\[
\kappa_{\alpha,\beta,t}=\frac{1}{2\alpha}\,\gamma\Big(\alpha+1,\alpha\ln\frac{\beta}{t}\Big),\quad\mbox{where}\quad \gamma(a,x)=\int_0^x y^{a-1}e^{-y}\,\dint y.
\]
Finally the average number of steps equals
\[
\mathbb{E}[N]=\frac{\alpha^\alpha}{\gamma\Big(\alpha+1,\alpha\ln\frac{\beta}{t}\Big)}\ \Big(\ln\frac{\beta}{t}\Big)^\alpha\ \Big(1 - \frac{t^\alpha}{\beta^\alpha}  \Big),\quad \mbox{for}\ t<\beta.    
\]
\end{rem}
\begin{Calculus} Let us do this calculus by using an integration by parts:
\[
\begin{array}{ll}
\kappa_{\alpha,\beta,t}& = \ds\int_0^{\rho_{\alpha,\beta,t}} u_{\alpha,\beta}(t,y)y^{2\alpha-1}\dint y\\
&=\ds\int_0^{\rho_{\alpha,\beta,t}} u_{\alpha,\beta}(t,y)\left(\frac{y^{2\alpha}}{2\alpha}\right)'\dint y\\
&= \left[\ds\frac{y^{2\alpha}}{2\alpha}u_{\alpha, \beta}(t,y)\right]_0^{\rho_{\alpha,\beta,t}} - \ds\int_0^{\rho_{\alpha,\beta,t}} \ds\frac{\partial u_{\alpha,\beta}}{\partial y} (t,y) \ds\frac{y^{2\alpha}}{2\alpha}\dint y.
\end{array}
\]

The first term on the last expression is 0 by the definition of $\rho_{\alpha,\beta,t}$. For the second one, by using the expression of $u_{\alpha,\beta}(t,y)$ and by performing the change of variable $\frac{y^2}{t}=x$ we get
\[
\begin{array}{ll}
\kappa_{\alpha,\beta,t}& = - \ds\int_0^{\rho_{\alpha,\beta,t}} \ds\frac{\partial u_{\alpha,\beta}}{\partial y} (t,y) \ds\frac{y^{2\alpha}}{2\alpha}\dint y\\
&=\ds\int_0^{\rho_{\alpha,\beta,t}}\ds\frac{1}{t^\alpha} \left(-\frac{2y}{t}\right)\exp\left(-\frac{y^2}{t}\right) \ds\frac{y^{2\alpha}}{2\alpha}\dint y\\
&= \ds\int_0^{\alpha \ln(\frac{\beta}{t})} \frac{1}{2\alpha} x^\alpha e^{-x}\dint x\\ 
&=
\ds\frac{1}{2\alpha}\,\gamma\Big(\alpha+1,\alpha\ln\frac{\beta}{t}\Big),\quad\mbox{where}\quad \gamma(a,x)=\int_0^x y^{a-1}e^{-y}\,\dint y.
\end{array}
\]
\end{Calculus}

\mathversion{bold}
\subsection*{The Bessel skeleton (non integer dimension $\delta>1$)}
\mathversion{normal}
 As already mentioned, our approach for the general case is based on Shiga-Watanabe's identity in order to split the simulation challenge into two parts: a Bessel process of integer dimension on one hand and a Bessel process of dimension less than $1$ on the other hand. In the sequel, for an easy identification of these two parts, we shall use for most of the parameters either the index $i$ corresponding to the integer part or the index $f$ for the fractional one. 

Let us fix two parameters $w_i\in]0,1[$ and $w_f\in]0,1[$ satisfying the following relation
\begin{equation}
\label{eq:relw}
w_f+2\sqrt{w_i}=1.
\end{equation}
Let us also define the general Bessel skeleton for a non integer dimension $\delta>1$. We need to introduce the following constants:
\[
\alpha_i:=\delta_i/2,\quad \alpha_f:=\delta_f/2, \quad \beta_i:=2ew_i\varepsilon^2/\delta_i
 \quad \mbox{and}\quad \beta_f:=2ew_f\varepsilon^2/\delta_f.
\]
We approximate a Bessel path, with starting value $y$, by constucting the following algorithm.

\begin{framed}
\centerline{\sc Bessel Skeleton $({\rm BeS})_\delta^w$ -- non integer dimension $\delta>1$}

\vspace*{0.2cm}
\begin{enumerate}
\item Let $(\ai_n)_{n\ge 1}$  be a sequence of i.i.d. random variables with gamma distribution ${\rm Gamma}(\nu_i+2,1/(\nu_i+1))$.
\item Let $(\af_n)_{n\ge 1}$  be a sequence of i.i.d. random variables with gamma distribution ${\rm Gamma}(\nu_f+2,1/(\nu_f+1))$.
\item Let $(V_n)_{n\ge 1}$ be a sequence of i.i.d. uniformly distributed random vectors on the boundary of the unitary sphere $\mathcal{S}^{\delta_i}$. We denote by $\pi_1(V_n)$ the projection on
the first coordinate. 
\end{enumerate}
The sequences 
 $(\ai_n)_{n\ge 1}$, $(\af_n)_{n\ge 1}$ and $(V_n)_{n\ge 1}$ are assumed to be independent. \\[5pt]
{\bf Initialization:} $n=0$, $y^\varepsilon_n=y$, $u_n^\varepsilon=0$, $s_n^\varepsilon=0$.\\[5pt]
{\bf Step 1.} Set $n\leftarrow n+1$.\\[5pt]
{\bf Step 2.} {\bf If} $\af_n-\ai_n<\ln\Big(\frac{w_f}{w_i}\Big)+\ln\frac{\delta_i}{\delta_f}$ {\bf then}\\[5pt]
\hspace*{2cm}\begin{minipage}{12cm}\begin{itemize}
\item Set $u_n^\varepsilon=\frac{\varepsilon^2w_i}{\delta_i}\,e^{1-\ai_n}$ and $\mathcal{Y}=\phi_{\delta_i,\varepsilon\sqrt{w_i}}(u_n^\varepsilon)$ 
\item Generate $\mathcal{Z}\sim ({\rm CD})^{2u_n^\varepsilon}_{\alpha_f,\beta_f}$ 
\end{itemize}\end{minipage}\\[5pt]
\hspace*{1.7cm}{\bf else} \\[5pt]
\hspace*{2cm}\begin{minipage}{12cm}\begin{itemize}
\item Set $u_n^\varepsilon=\frac{\varepsilon^2w_f}{\delta_f}\,e^{1-\af_n}$ and $\mathcal{Z}=\phi_{\delta_f,\varepsilon\sqrt{w_f}}(u_n^\varepsilon)$
\item Generate $\mathcal{Y}\sim ({\rm CD})^{2u_n^\varepsilon}_{\alpha_i,\beta_i}$
\end{itemize}\end{minipage}\\[12pt]
{\bf Step 3.} Set $y_n^\varepsilon=\Big((y_{n-1}^\varepsilon)^2+2y_{n-1}^\varepsilon\pi_1(V_n)\,\mathcal{Y}+\mathcal{Y}^2+\mathcal{Z}^2\Big)^{1/2}$ and $s_n^\varepsilon=s_{n-1}^\varepsilon+u_n^\varepsilon$.\\[5pt]
\hspace*{1.5cm} Return to Step 1.\\[5pt]
\noindent
{\bf  Definition:} The \emph{Bessel skeleton} $({\rm BeS})_\delta^w$ corresponds to $\Big((u_n^\varepsilon,s_n^\varepsilon)_{n\ge 1},(y_n^\varepsilon)_{n\ge 0}\Big)$.
\end{framed}

The algorithm is based on the construction of a sequence of points $(s_n^\varepsilon,y_n^\varepsilon)_{n\ge 0}$ which essentially permit to emphasize an approximated Bessel path. This sequence is obtained in a Markovian step by step procedure. With a starting time and location $(s_n^\varepsilon,y_n^\varepsilon)$, corresponding to the value of the $\delta$-dimensional Bessel process, we associate two sets composed of a starting point and a spheroid: one intended for a Bessel process of integer dimension and the other one for a Bessel process of fractional dimension. These two paths have been carefully observed until one of them exits from its spheroid. At that random time $s_{n+1}$, both paths are stopped and a combination of their position at that stage permits to compute $y_{n+1}$. To sum up, each step of the algorithm starts with a splitting of the Bessel paths and ends up with a regluing procedure. The sequence  $(s_n^\varepsilon,y_n^\varepsilon)_{n\ge 0}$ is crucial for the path approximation as pointed out in the following statement. 

\begin{thm}\label{thm:Bessne}
Let $\varepsilon>0$ and let $w=(w_i,w_f)\in]0,1[^2$ be a couple of weights satisfying the condition \eqref{eq:relw}. Consider a Bessel skeleton $({\rm BeS})_{\delta}^w$ with a non integer dimension $\delta>1$. Then $y^\varepsilon_t=\sum_{n\ge 0}y_n^\varepsilon1_{\{  s_n^\varepsilon \le t<s^\varepsilon_{n+1}\}}$ is an $\varepsilon$-strong approximation of the Bessel paths starting in $y$, solution of \eqref{eq:def:Be}. 
The number of approximation points $N_T^\varepsilon$ on the fixed interval $[0,T]$ satisfies:
\begin{equation}\label{eq:thm:2}
\lim_{\varepsilon\to 0}\varepsilon^2\,\mathbb{E}[N^\varepsilon_T]=\frac{T}{\mu}:=\frac{T\delta_i}{ew_i}\ \mathcal{F}\Big(\frac{w_f}{w_i}\frac{\delta_i}{\delta_f},\nu_f+2,\frac{1}{\nu_f+1},\nu_i+2,\frac{1}{\nu_i+1}\Big)^{-1},
\end{equation}
where $\mathcal{F}(x,a,\lambda,b,\mu)=\mathbb{E}[\min(x e^{-A},e^{-B})]$. Here $A$ and $B$ stand for two independent Gamma distributed random variables with parameters (shape $a$ and scale $\lambda$) and, respectively $(b,\mu)$. 
Moreover the following CLT is observed:
\[
\ds\lim_{\varepsilon \to 0}\frac{1}{\varepsilon\sqrt
T}\frac{\mu^{3/2}\delta_i}{ew_i\sigma}\Big(  \varepsilon^2\, N^\varepsilon_T-\frac{T}{\mu}\Big) = G \quad \mbox{in distribution},
\]
with $G$ a $\mathcal{N}(0,1)$ standard Gaussian random variable and 
\begin{equation}
\label{def:sigma2}
\sigma^2=\mathcal{F}\Big(\frac{w_f^2}{w_i^2}\frac{\delta_i^2}{\delta_f^2},\nu_f+2,\frac{2}{\nu_f+1},\nu_i+2,\frac{2}{\nu_i+1}\Big)-\mathcal{F}\Big(\frac{w_f}{w_i}\frac{\delta_i}{\delta_f},\nu_f+2,\frac{1}{\nu_f+1},\nu_i+2,\frac{1}{\nu_i+1}\Big)^2.
\end{equation}
\end{thm}
\begin{corollary}\label{cor:bessne} There exists $\varepsilon_0>0$, such that the average number of approximation points $N_T^\varepsilon$, on the fixed interval $[0,T]$, satisfies:
\begin{equation}\label{eq:upperb}
\varepsilon^2\mathbb{E}[N_T^\varepsilon]\le \frac{T}{e}\max\Big(\frac{\delta_i}{w_i},\frac{\delta_f}{w_f}\Big)\  \Big( \frac{\nu_i+2}{\nu_i+1} \Big)^{\nu_i+2}\Big( \frac{\nu_f+2}{\nu_f+1} \Big)^{\nu_f+2},\quad \mbox{for all}\quad \varepsilon\le \varepsilon_0.
\end{equation}
The right hand side of \eqref{eq:upperb} can be minimized with the optimal choice: $w_i=\Big( \frac{\sqrt{\delta_i\delta}-\delta_i}{\delta_f} \Big)^2$.
\end{corollary}
\begin{proof}[Proof of Corollary \ref{cor:bessne}] The statement is a direct consequence of the convergence result \eqref{eq:thm:2}, combined with the properties of the function $\mathcal{F}$. More precisely, for $x\in(0,1]$, the independence of the variates $A$ and $B$ leads to
\begin{align*}
\mathcal{F}(x,a,\lambda,b,\mu)&=\mathbb{E}[\min(x e^{-A},e^{-B})]=\mathbb{E}[\exp\{-\max(A-\log(x),B)\}]\\
&> \mathbb{E}[\exp\{-(A+B-\log(x))\}]=x\mathbb{E}[e^{-A}]\mathbb{E}[e^{-B}]\\
&=x\mathcal{L}_A(1)\mathcal{L}_B(1)=x(1+\lambda)^{-a}(1+\mu)^{-b},
\end{align*}
 since $\max(A-\log(x),B)<A+B-\log(x)$ as soon as $A>0$ and $B>0$, that is almost surely.
Here $\mathcal{L}_{A}(s)$ stands for the Laplace transform of the variate $A$, that is
\(
\mathcal{L}_{A}(s)=(1+\lambda s)^{-a}
\). Moreover, if $x\ge 1$, then similar computations lead to
\begin{align*}
\mathcal{F}(x,a,\lambda,b,\mu)&=x\mathbb{E}[\min(e^{-A},x^{-1}e^{-B})]> x\cdot x^{-1}(1+\lambda)^{-a}(1+\mu)^{-b}=(1+\lambda)^{-a}(1+\mu)^{-b}.
\end{align*}
Consequently, for any $x>0$, we get
\begin{align}\label{top}
\mathcal{F}(x,a,\lambda,b,\mu)^{-1}&< \max(1,x^{-1})\,(1+\lambda)^{a}(1+\mu)^{b}.
\end{align}
Combining \eqref{top} with the limiting value \eqref{eq:thm:2} leads therefore to the announced upper-bound \eqref{eq:upperb}.
\end{proof}

\begin{proof}[Proof of Theorem \ref{thm:Bessne}] The structure of the proof is similar to Theorem \ref{eq:thm:1}. First we replace the paths of the Bessel process by some other paths with the same distribution. Then, on the new paths, we introduce a skeleton. Finally we count the number of points needed to cover a deterministic time interval $[0,T]$.\\[5pt]
\noindent \emph{Step 1: $\varepsilon$-strong approximation.} \\[5pt]
Let us consider a Bessel process of non integer dimension $\delta>1$, that is the solution of equation \eqref{eq:def:Be} with initial value $y\ge 0$. Let us denote the distribution of the squared process by $\mathbb{Q}^{\delta,y^2}$. We recall that the dimension can be decomposed as follows:  $\delta=\delta_i+\delta_f$ where $\delta_i=\lfloor \delta\rfloor$. Using the identity in law pointed out by Shiga and Watanabe \cite{Shiga}, we obtain:
\begin{equation}
\mathbb{Q}^{\delta,y^2}=\mathbb{Q}^{\delta_i,y^2}\star \mathbb{Q}^{\delta_f,0},
\label{eq:shiga}
\end{equation}
where $\star$ stands for the convolution of the probability distributions. Thus, by introducing two independent Bessel processes $\overline{Z}$ and $\widehat{Z}$, one of integer dimension $\delta_i$ starting in $y$: $(\overline{Z}_t(y))_{t\ge 0}$, and the other of non integer dimension $\delta_f$, starting in $0$: $(\widehat{Z}_t(0))_{t\ge 0}$ (when the starting value is equal to $0$, we shall drop the dependence for notational simplicity), then \eqref{eq:shiga} leads to the identity
\begin{equation}
(Z^{\delta,y}_t)_{t\ge 0}\overset{(d)}{=}\Big(\overline{Z}_t(y)^2+\widehat{Z}_t(0)^2\Big)^{1/2}_{t\ge 0}=:\Big(\overline{Z}_t(y)^2+\widehat{Z}_t^2\Big)^{1/2}_{t\ge 0}.\label{eq:shiga2}
\end{equation}
Bessel processes with integer dimension $\delta_i$ play an important role since they can be represented as the norm of the $\delta_i$-dimensional Brownian motion. We just note that the standard Brownian motion  $(W_t)$ is rotational invariant and moreover:
\[
(W_t)_{t\ge 0}=\Big(\Theta_t\cdot\Vert W_t\Vert\Big)_{t\ge 0},
\]
where $(\Theta_t)_{t> 0}$ is a continuous stochastic process valued in the unitary sphere $\mathcal{S}^{\delta_i}$, independent of $(\Vert W_t\Vert)_{t\ge 0}$. We fix $\Theta_0=0$ (not continuous for $t=0$) and observe that $\Theta_t$ is uniformly distributed at any fixed time $t>0$. 

\noindent Let us denote $\underline{y}=(y,0,\ldots,0)\in\mathbb{R}^{\delta_i}$. We deduce that 
\begin{align}
(\overline{Z}_t(y))_{t\ge 0}&\overset{(d)}{=}\Big\Vert \underline{y}+\Theta_t\Vert W_t\Vert\,\Big\Vert_{t\ge 0}=\Big( y^2+2y\,\pi_1(\Theta_t)\Vert W_t\Vert+\Vert W_t\Vert^2 \Big)^{1/2}_{t\ge 0},
\label{eq:entier}
\end{align}
where $\pi_1$ corresponds to the projection on the first coordinate. Combining \eqref{eq:shiga2} and \eqref{eq:entier} leads to
\begin{equation}\label{eq:decomp}
(Z^{\delta,y}_t)_{t\ge 0}\overset{(d)}{=}\mathcal{X}_t:=\Big( y^2+2y\,\pi_1(\Theta_t)\overline{Z}_t +\overline{Z}_t^2 +\widehat{Z}_t^2 \Big)^{1/2}_{t\ge 0},
\end{equation}
where the processes $(\overline{Z}_t)_{t\ge 0}$, $(\widehat{Z}_t)_{t\ge 0}$ and $(\Theta_t)_{t\ge 0}$ are independent. Using the strong Markov property of the Bessel process, we can propose a more complex identity. If $s_1$ is a stopping time with respect to the filtration $\mathcal{F}^{(1)}:=(\mathcal{F}_t^{(1)})_{t\ge 0}$ induced by $(W,\widehat{Z})$  (also denoted in the sequel  $(W^{(1)},\widehat{Z}^{(1)})$) then the conditional distribution of $(\mathcal{X}_{s_1+t})_{t\ge 0}$ given $\mathcal{F}_{s_1}$ is identical to the distribution 
\[
(\mathcal{X}_t^{(2)})_{t\ge 0}:=\Big( \mathcal{X}_{s_1}^2+2\mathcal{X}_{s_1}\,\pi_1(\Theta_t^{(2)})\overline{Z}_t^{(2)} +(\overline{Z}_t^{(2)})^2 +(\widehat{Z}_t^{(2)})^2 \Big)^{1/2}_{t\ge 0},
\]
where $((\overline{Z}^{(k)},\widehat{Z}^{(k)},\Theta^{(k)})_{t\ge 0})_{k\ge 2}$ is a family of independent copies of $(\overline{Z}^{(1)},\widehat{Z}^{(1)},\Theta^{(1)})_{t\ge 0}$. So we can build a particular stochastic process $(\mathcal{X}_t)_{t\ge 0}$ combining $\mathcal{X}$ (also denoted $\mathcal{X}^{(1)}$) and $\mathcal{X}^{(2)}$ by the following identity
\[
\overline{\mathcal{X}}_t^{(2)}:=\mathcal{X}^{(1)}_t1_{\{t< s_1\}}+\mathcal{X}^{(2)}_{t-s_1}1_{\{t\ge s_1\}},\quad t\ge 0.
\]
Let us note that both $(\overline{\mathcal{X}}_t)_{t\ge 0}$ and $(Z^{\delta,y}_t)_{t\ge 0}$ are identically distributed. Let us go on with the modification of the process. To that end, we denote by $(\mathcal{F}_t^{(2)})_{t\ge 0}$ the filtration generated by the following stochastic processes: $(W^{(1)},\widehat{Z}^{(1)})_{t\wedge s_1}$ and $(W^{(2)},\widehat{Z}^{(2)})_{(t-s_1)\vee 0}$. For any $\mathcal{F}^{(2)}$-stopping time  $s_2>s_1$, we can define
\[
\overline{\mathcal{X}}_t^{(3)}:=\mathcal{X}^{(1)}_t1_{\{t< s_1\}}+\mathcal{X}^{(2)}_{t-s_1}1_{\{s_1\le t< s_2\}}+\mathcal{X}^{(3)}_{t-s_2}1_{\{t\ge  s_2\}},\quad t\ge 0,
\]
where $\mathcal{X}^{(3)}$ is defined by 
\[
(\mathcal{X}_t^{(3)})_{t\ge 0}:=\Big( (\mathcal{X}^{(2)}_{s_2-s_1})^2+2\mathcal{X}^{(2)}_{s_2-s_1}\,\pi_1(\Theta_t^{(3)})\overline{Z}_t^{(3)} +(\overline{Z}_t^{(3)})^2 +(\widehat{Z}_t^{(3)})^2 \Big)^{1/2}_{t\ge 0}.
\]
The procedure continues step by step in this way. For any increasing sequence of stopping time $(s_n)_{n\ge 1}$, satisfying $\lim_{n\to\infty}s_n=+\infty$, we construct the stochastic process:
\begin{equation}
\label{eq:def:procinfini}
\overline{\mathcal{X}}_t^{(\infty)}=\mathcal{X}_{t-s_n}^{(n+1)},\quad \mbox{if}\ s_n\le t<s_{n+1},
\end{equation}
with the definition 
\begin{equation}\label{eq:def:procinfini2}
(\mathcal{X}_t^{(n+1)})_{t\ge 0}:=\Big( (\mathcal{X}^{(n)}_{s_n-s_{n-1}})^2+2\mathcal{X}^{(n)}_{s_n-s_{n-1}}\,\pi_1(\Theta_t^{(n+1)})\overline{Z}_t^{(n+1)} +(\overline{Z}_t^{(n+1)})^2 +(\widehat{Z}_t^{(n+1)})^2 \Big)^{1/2}_{t\ge 0}.
\end{equation}
By construction, we observe that $(\overline{\mathcal{X}}_t^{\infty})_{t\ge 0}$ and $(Z^{\delta,y}_t)_{t\ge 0}$ are identically distributed. Since the definition of the $\varepsilon$-strong approximation only depends on the distribution of the stochastic process, it suffices therefore to point out an approximation of $(\overline{\mathcal{X}}_t^{\infty})_{t\ge 0}$ in order to prove the statement.\\[5pt]
\noindent\emph{Step 2: Bessel skeleton}\\[5pt]
Let us now consider a particular increasing family of stopping times. Let $w\in]0,1[$ be a fixed parameter. We define $\overline{\tau}_n$ (respectively $\widehat{\tau}_n$), the first passage time of the Bessel process $(\overline{Z}_t^{(n)})_{t\ge 0}$ (resp. $(\widehat{Z}_t^{(n)})_{t\ge 0}$), through the curved boundary $\phi_{\delta_i,\varepsilon\sqrt{w_i}}$ (resp. $\phi_{\delta_f,\varepsilon\sqrt{w_f}}$), defined in \eqref{def:phi}. We construct a new stopping time $u_n$ and the associated cumulative time $s_n$, as follows:
\begin{equation}\label{defdeu}
u_n:=\overline{\tau}_n \wedge \widehat{\tau}_n\quad\mbox{and}\quad s_{n}=s_{n-1}+u_n,\quad n\ge 1,
\end{equation}
with the initial value $s_0=0$. The sequence of stopping times $(s_n)_{n\ge 0}$ satisfies the conditions developed in the previous paragraph Step 1. We can therefore construct the continuous process $(\mathcal{X}^\infty_t)_{t\ge 0}$ using \eqref{eq:def:procinfini}--\eqref{eq:def:procinfini2} and the particular sequence $(s_n)_{n\ge 0}$, just described. Since the maximal value of the curved boundary $\phi_{\delta,\varepsilon}$ equals $\varepsilon$, we can emphasize a crucial upper-bound of the difference  $\mathcal{D}^{(n)}_t:=|\overline{\mathcal{X}}^{\infty}_t-\overline{\mathcal{X}}^{\infty}_{s_n}|$. For any $s_n\le t<s_{n+1}$, 
\begin{align*}
\mathcal{D}_t^{(n)}=\Big|\Big((\mathcal{X}^{(n)}_{s_n-s_{n-1}})^2+2\mathcal{X}^{(n)}_{s_n-s_{n-1}}\,\pi_1(\Theta_t^{(n+1)})\overline{Z}_t^{(n+1)} +(\overline{Z}_t^{(n+1)})^2 +(\widehat{Z}_t^{(n+1)})^2 \Big)^{1/2}-\mathcal{X}^{(n)}_{s_n-s_{n-1}} \Big|\\
=\Big|\Big((\mathcal{X}^{(n)}_{s_n-s_{n-1}}+\pi_1(\Theta_t^{(n+1)})\overline{Z}_t^{(n+1)} )^2+(\overline{Z}_t^{(n+1)})^2(1-\pi_1^2(\Theta_t^{(n+1)})) +(\widehat{Z}_t^{(n+1)})^2 \Big)^{1/2}-\mathcal{X}^{(n)}_{s_n-s_{n-1}} \Big|.
\end{align*}
\begin{Calculus} Ici une petite remarque sur une inégalité très simple. Soient $a$ et $b$ deux nombres positifs ou nuls, on peut alors montrer que  $|\sqrt{a+b}-x|\le |\sqrt{a}-x|+\sqrt{b}$. En effet
\[
-\sqrt{b}-|\sqrt{a}-x|\le -|\sqrt{a}-x|\le \sqrt{a}-x \le  \sqrt{a+b}-x 
\]
et
\[
\sqrt{a+b}-x \le \sqrt{a}+\sqrt{b}-x\le |\sqrt{a}-x|+\sqrt{b}.
\]
On obtient également une autre inégalité : $\Big| |a+b|-|a|\Big|\le |b|$. En effet, on a d'une part
\[
|a+b|-|a|\le |a|+|b|-|a|=|b|,
\]
et d'autre part
\[
|a+b|-|a|=|a+b|-|a+b-b|\ge |a+b|-|a+b|-|b|=-|b|.
\]
\end{Calculus}
\noindent Let us consider $a$ and $b$ two non negative numbers, then for any $x\in\mathbb{R}$, we have $|\sqrt{a+b}-x|\le |\sqrt{a}-x|+\sqrt{b}$, $\sqrt{a+b}\le \sqrt{a}+\sqrt{b}$ and finally $| |a+b|-|a||\le |b|$. Applying to the previous expression of $\mathcal{D}_t^{(n)}$, we obtain
\begin{align*}
\mathcal{D}^{(n)}_t&\le \Big| |\mathcal{X}^{(n)}_{s_n-s_{n-1}}+\pi_1(\Theta_t^{(n+1)})\overline{Z}_t^{(n+1)} |-\mathcal{X}^{(n)}_{s_n-s_{n-1}}\Big|+\sqrt{w_i\varepsilon^2+w_f\varepsilon^2}\\
&\le | \pi_1(\Theta_t^{(n+1)})\overline{Z}_t^{(n+1)} |+\varepsilon\sqrt{w_i+w_f}\le \varepsilon(\sqrt{w_i}+ \sqrt{w_i+w_f})=\varepsilon.
\end{align*}
The last equality is a consequence of the particular relation between $w_i$ and $w_f$ introduced in \eqref{eq:relw}. We deduce therefore that the stochastic process defined by $\widehat{y}_t=\sum_{n\ge 0}\overline{\mathcal{X}}^\infty_{s_n}1_{\{s_n\le t<s_{n+1}\}}$, is an $\varepsilon$-strong approximation of the Bessel paths (see Definition \ref{def}). In order to prove the statement of Theorem \ref{thm:Bessne}, it suffices to check that $(y_t^\varepsilon)_{t\ge 0}$ defined in the statement and $(\widehat{y}_t)_{t\ge 0}$, are identically distributed. Let us therefore describe the joint distribution of $(u_n)_{n\ge 1}$, $(s_n)_{n\ge 0}$ and $(\overline{\mathcal{X}}^\infty_{s_n})_{n\ge 1}=(\mathcal{X}^{(n)}_{s_{n}-s_{n-1}})_{n\ge 1}$ and compare it to the Bessel skeleton.
\begin{itemize}
\item Using the definition of the stopping times $u_n$ in \eqref{defdeu}, we observe that $(u_n)_{n\ge 0}$ is a sequence of independent and identically distributed random variables. 
Moreover, on one hand, the distribution of the first passage time $\overline{\tau}_n$ is identical to that of $\frac{\varepsilon^2 w_i}{\delta_i}\,e^{1-\ai}$, where $\ai$ stands for a Gamma distributed r.v of parameters $\nu_i+2$ and $1/(\nu_i+1)$ (see for instance \cite{Deaconu-Herrmann-2017}). 
On the other hand, $\widehat{\tau}_n$ and $\frac{\varepsilon^2 w_f}{\delta_f}\,e^{1-\af}$ are identically distributed. Here $\af$ corresponds to Gamma distributed r.v. with parameters $\nu_f+2$ and $1/(\nu_f+1)$. The stopping time $u_n$ is the minimum of these two first passage times and matches the stopping time $u_n^\varepsilon$ introduced in Algorithm $({\rm BeS})^w_\delta$.
Consequently $(s_n)_{n\ge 0}$ and $(s_n^\varepsilon)_{n\ge 0}$ are identically distributed. 
\item Let us now describe the sequence $(\mathcal{X}^{(n)}_{s_{n}-s_{n-1}})_{n\ge 1}$. It is defined recursively by \eqref{eq:def:procinfini2}. In this equation, we need to know the value of three random variables: $\Theta_{u_{n}}$, $\overline{Z}^{(n)}_{u_{n}}$, $\widehat{Z}^{(n)}_{u_{n}}$. Since $u_n$ is only linked to stopping times defined on the processes $\overline{Z}^{(n)}$ and $\widehat{Z}^{(n)}$, which are independent from $\Theta$, and since $\Theta_t$ is uniformly distributed for any $t>0$, we obtain that $\Theta_{u_{n}}$ is uniformly distributed on the sphere and independent of both $\overline{Z}^{(n)}_{u_{n}}$ and $\widehat{Z}^{(n)}_{u_{n}}$. Moreover the definition \eqref{eq:defdeu} implies to take into account two different cases: either $u_n=\overline{\tau}_n<\widehat{\tau}_n$ or $u_n=\widehat{\tau}_n<\overline{\tau}_n$. In the first case, we have, on the event $u_n=t$, $\overline{Z}^{(n)}_{u_{n}}=\phi_{\delta_i,\varepsilon\sqrt{w_i}}(t)$ and the distribution of $\widehat{Z}^{(n)}_{u_{n}}$ corresponds to $({\rm CD})_{\alpha_f,\beta_f}^{2t}$ as announced in Lemma \ref{lem:identlaw}: a Bessel process conditioned not to have reach a curved boundary. In the second case, we observe the reverse situation: on the event $u_n=t$, $\widehat{Z}^{(n)}_{u_{n}}=\phi_{\delta_f,\varepsilon\sqrt{w_f}}(t)$ and the distribution of $\overline{Z}^{(n)}_{u_{n}}$ corresponds to $({\rm CD})_{\alpha_i,\beta_i}^{2t}$ as announced in Lemma \ref{lem:identlaw}. To sum up, 
\[
(\Theta_{u_{n}}, \overline{Z}^{(n)}_{u_{n}}, \widehat{Z}^{(n)}_{u_{n}})\overset{(d)}{=}(V_n,\mathcal{Y},\mathcal{Z}),
\]
where $V_n$,$\mathcal{Y}$ and $\mathcal{Z}$ correspond to the variables introduced in Algorithm $({\rm BeS})_{\delta}^w$. Due to \eqref{eq:def:procinfini2}, we deduce quite easily that $(y_t^\varepsilon)_{t\ge 0}$ and $(\widehat{y}_t)_{t\ge 0}$ are identically distributed. We conclude that $(y_t^\varepsilon)_{t\ge 0}$ is an $\varepsilon$-strong approximation of the Bessel paths.
\end{itemize}
\noindent \emph{Step 3: Number of points necessary to cover the time interval $[0,T]$.}\\[5pt]
The arguments for the description of the number of points have already been introduced in the proof of Theorem \ref{thm:Bess1}. We introduce $(\widehat{N}_t)_{t\ge 0}$ a Poisson process with independent and identically distributed arrivals $(M_n)_{n\ge 1}$ where 
\[
M_n=\Big(\frac{ w_i}{\delta_i}\,e^{1-\ai_n}\Big)\wedge \Big(\frac{w_f}{\delta_f}\,e^{1-\af_n}\Big),
\]
with $\ai$ and $\af$ defined in Algorithm $({\rm BeS})_\delta^w$. We denote $\mu=\mathbb{E}[M_1]$. The classical asymptotic result holds:
\begin{equation}
\label{eq:limitpoiss}
\lim_{t\to \infty}\frac{\mathbb{E}[\widehat{N}_t]}{t}=\frac{1}{\mu}= \frac{\delta_i}{e w_i} \mathcal{F}\Big( \frac{w_f}{w_i}\frac{\delta_i}{\delta_f},\nu_f+2,\frac{1}{\nu_f+1},\nu_i+2,\frac{1}{\nu_i+1} \Big)^{-1},
\end{equation}
 where $\mathcal{F}$ is defined in the statement of Theorem \ref{thm:Bessne}. The mean of $M_1$ plays an important role in the limit so do the variance for the confidence interval. Due to the scaling property of the Gamma distribution, we notice that ${\rm Var}(M_1)=\frac{e^2w_i^2}{\delta_i^2}\sigma^2$, where $\sigma^2$ is defined by \eqref{def:sigma2}. The central limit theorem, applied in the counting process context, leads to 
\[
\lim_{t\to\infty}\sqrt{\frac{t\mu^3\delta_i^2}{e^2w_i^2\sigma^2}}\Big( \frac{\widehat{N}_t}{t}-\frac{1}{\mu} \Big)=G\quad \mbox{in distribution,}
\]
where $G$ is a $\mathcal{N}(0;1)$ standard Gaussian variate. Let us observe that the number of approximation points $N^\varepsilon_T$ is directly linked in distribution to the Poisson process just defined. More exactly, we have $N^\varepsilon_T\overset{(d)}{=}\widehat{N}_{\frac{T}{\varepsilon^2}}$, which gives directly the statement: the limit with respect to the time variable is replaced by the limit with respect to the parameter $\varepsilon$.
\end{proof}

\section{Related processes and numerical illustration}\label{sec:num}
\subsection{Numerics: Bessel processes}
Let us first illustrate the strong approximation of Bessel processes. We choose to observe the paths on some given time interval $[0,T]$. In particular, we are able to present a Bessel skeleton and the corresponding upper and lower bounds for some small precision value $\varepsilon$. 

\begin{figure}[h]
\centering
\includegraphics[width=10cm]{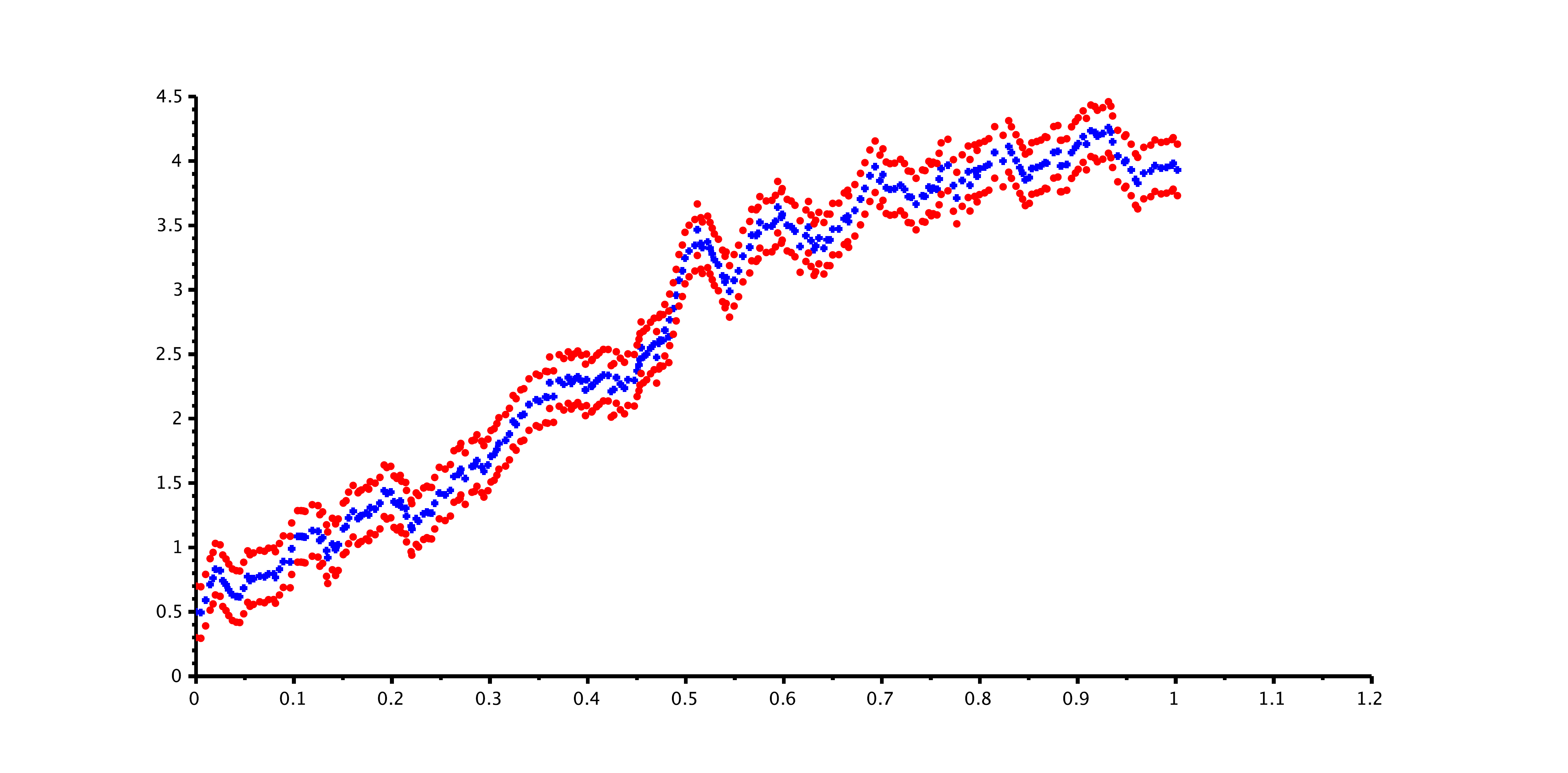}\hspace*{-0.7cm}\includegraphics[width=7cm]{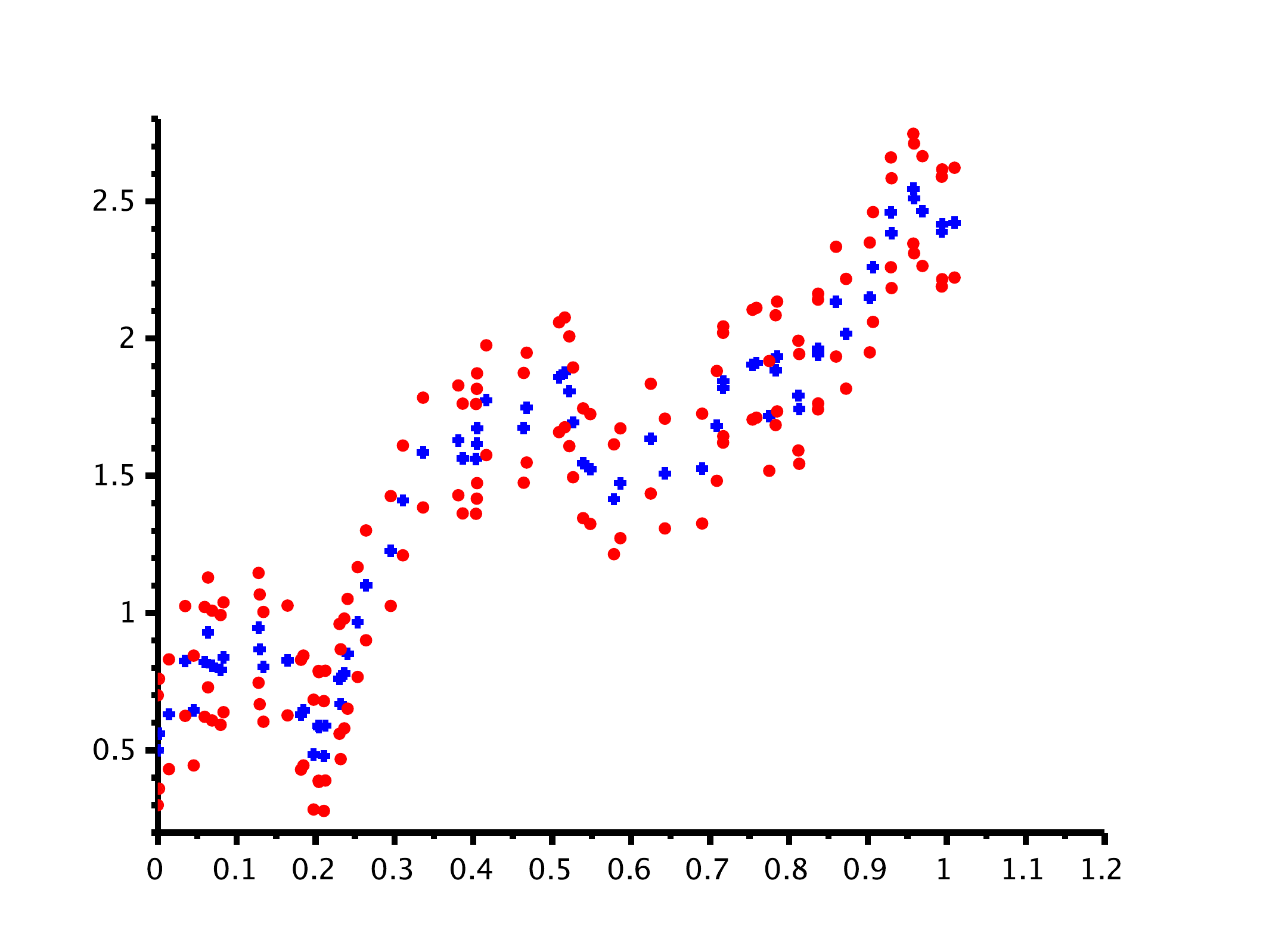}
\caption{\small A trajectory of a Bessel process either of dimension 10 (left) or of dimension 2 (right) represented by its squeleton (crosses) and the lower and upper almost sure bounds on $[0,1]$ with $y_0=0.5$, $\varepsilon=0.2$}
\label{fig:my_label1}
\end{figure}

In Figure \ref{fig:my_label1}, the skeletons correspond for instance to a Bessel process either of dimension $10$ or of dimension $2$. As we can observe, the variations of the skeleton when time elapses are obviously smaller than the limit $\varepsilon$ chosen for the approximation. We can interpret this as: even if the maximal size of the spheroids corresponds to this specific value $\varepsilon$, the difference between the values of two successive points of the Bessel skeleton is not often near to the maximum. Indeed the spheroid is applied to the $\delta$-dimensional Brownian motion in a first step, and then, in a second step, a random projection is applied, see the algorithm  $({\rm BeS})_\delta$ for integer dimensions. So denoting by $\tau$ the first Brownian exit time of the spheroid \eqref{def:phi} and by $p_\tau$ its probability density function, we can compute the following average size   
\begin{align*}
\mathbb{E}[\phi_{\delta,\varepsilon}(\tau)]&=\int_0^{\frac{e\varepsilon^2}{\delta}}\phi_{\delta,\varepsilon}(t)\,p_\tau(t)\,\dint t=\int_0^{\frac{e\varepsilon^2}{\delta}}\sqrt{\delta t\ln\Big( \frac{e\varepsilon^2}{\delta t} \Big)}\frac{1}{t\Gamma(\delta/2)}\left( \frac{\delta^2t}{2e\varepsilon^2}\,\ln\Big( \frac{e\varepsilon^2}{\delta t} \Big) \right)^{\delta/2}\,\dint t\\
&=\varepsilon\frac{\sqrt{e}}{\Gamma(\delta/2)}\Big( \frac{\delta}{2} \Big)^{\delta/2}\int_0^1\frac{1}{u}\,\Big( u\ln\frac{1}{u} \Big)^{(\delta+1)/2}\,\dint u=:\varepsilon \eta(\delta).
\end{align*}
We can evaluate this last integral
\begin{align*}
\int_0^1\frac{1}{u}\,\Big( u\ln\frac{1}{u} \Big)^{(\delta+1)/2}\,\dint u &= \Big(\frac{2}{\delta+1}\Big)^{\frac{\delta+3}{2}} \Gamma\left(\frac{\delta+3}{2}\right),
\end{align*}
and by using the properties of the Gamma function obtain the explicit form:
\begin{equation*}
\eta(\delta) = \sqrt{ 2\pi e}
\frac{\Gamma(\delta)}
{\left[\Gamma\left(\frac{\delta}{2}\right)\right]^2} \frac{\delta^{\frac{\delta}{2}}}{(\delta+1)^{\frac{\delta+1}{2}}}2^{1-\delta}.
\end{equation*}

\noindent\begin{minipage}{7.5cm}
The average is obviously proportional to $\varepsilon$ and the constant $\eta(\delta)$ can be evaluated easily. We can observe that $\eta$ is a non decreasing function of the dimension $\delta$ on the interval $[2,+\infty)$ starting with an estimated value $\eta(2)=0.7953$. This function is represented on the opposite figure. Let us note that for high dimensions the average size $\mathbb{E}[\phi_{\delta,\varepsilon}(\tau)]$ is close to $\varepsilon$, which is the optimal size for the strong approximation procedure. 
\end{minipage}\hspace*{1.5cm}
\begin{minipage}{7cm}
\centerline{\includegraphics[width=8cm]{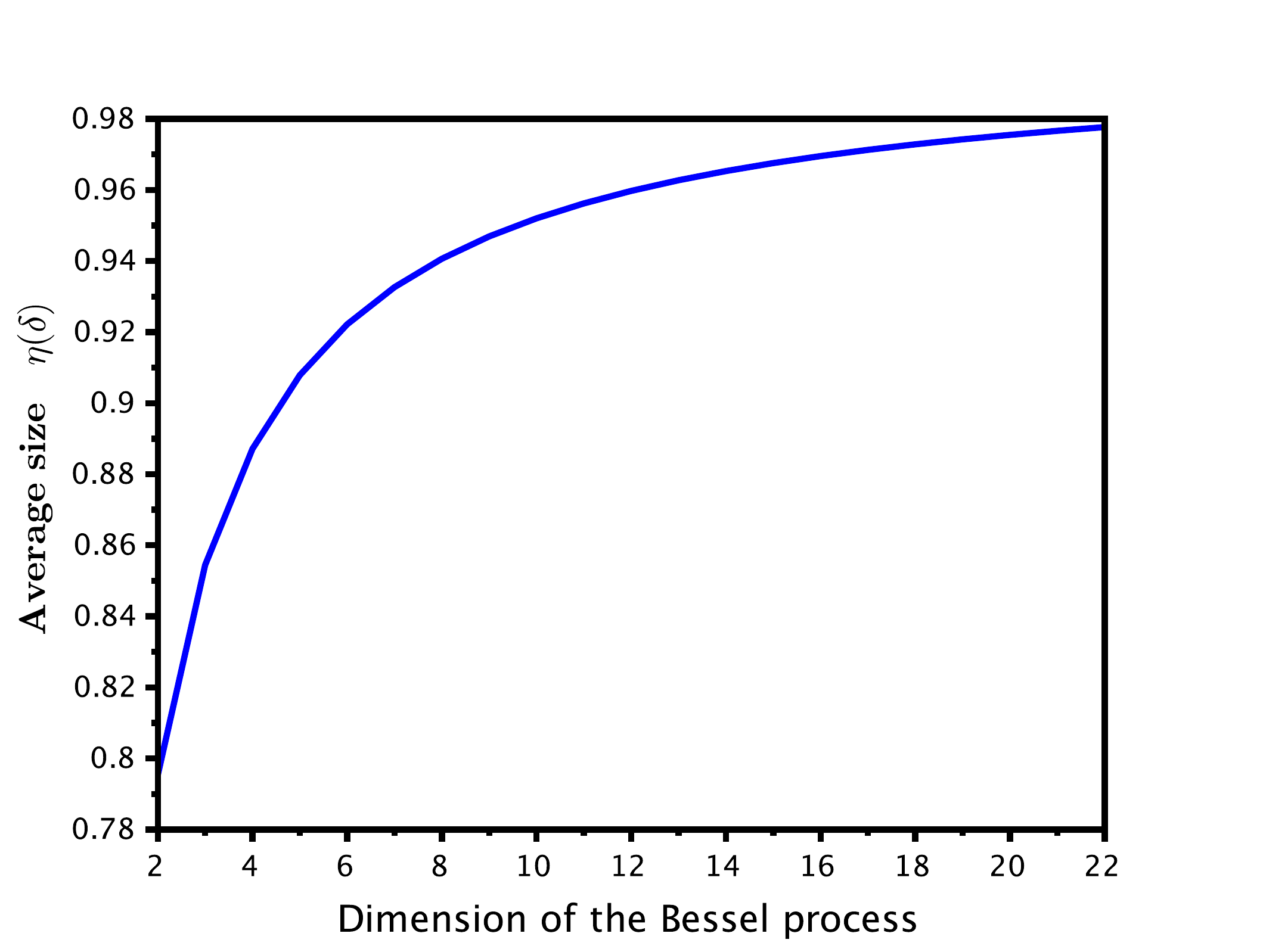}}
\end{minipage}\\[4pt]
On the one hand, the increments of the Bessel skeleton depend on the Brownian exit time of the spheroid. On the other hand, they are also strongly related to the projection on the first coordinate of a random variable $V$ uniformly distributed on the sphere of dimension $\delta$: $\pi_1(V)$. For $\delta>2$, using the spherical coordinates, we obtain
\begin{align}
\label{eq:pi}
\mathbb{E}[|\pi_1(V)|]&=\frac{2^{\delta}}{(\delta-1)(2\pi)^{\delta/2}}\left( \int_0^\infty r^{\delta-1}e^{-\frac{r^2}{2}}\,\dint r \right)\times \prod_{k=0}^{\delta-3}W_{k},
\end{align}
where $W_k$ stands for Wallis' integrals $W_n:=\int_0^{\pi/2}\sin^n(x)\,\dint x$. 
Let us note that the integral appearing in \eqref{eq:pi} can be related to the moments of a standard Gaussian variate. \\[4pt]
\begin{minipage}{7.5cm}
We deduce that
\[
\int_0^\infty r^{2k}e^{-\frac{r^2}{2}}\,\dint r=\sqrt{\frac{\pi}{2}}\,\frac{(2k)!}{2^kk!}
\]
and
\[
\int_0^\infty r^{2k+1}e^{-\frac{r^2}{2}}\,\dint r=2^kk!.
\]
We can therefore compute the average size of the projection which of course depends on the dimension. Let us just note that the particular dimension $\delta=2$ leads to $\mathbb{E}[|\pi_1(V)|]=\frac{2}{\pi}\approx  0.6366$. The opposite figure gives this dependence: for large dimensions the projection procedure reduces the difference between two successive points of the skeleton.
\end{minipage}\hspace*{1.5cm}
\begin{minipage}{8.5cm}
\centerline{\includegraphics[width=9cm]{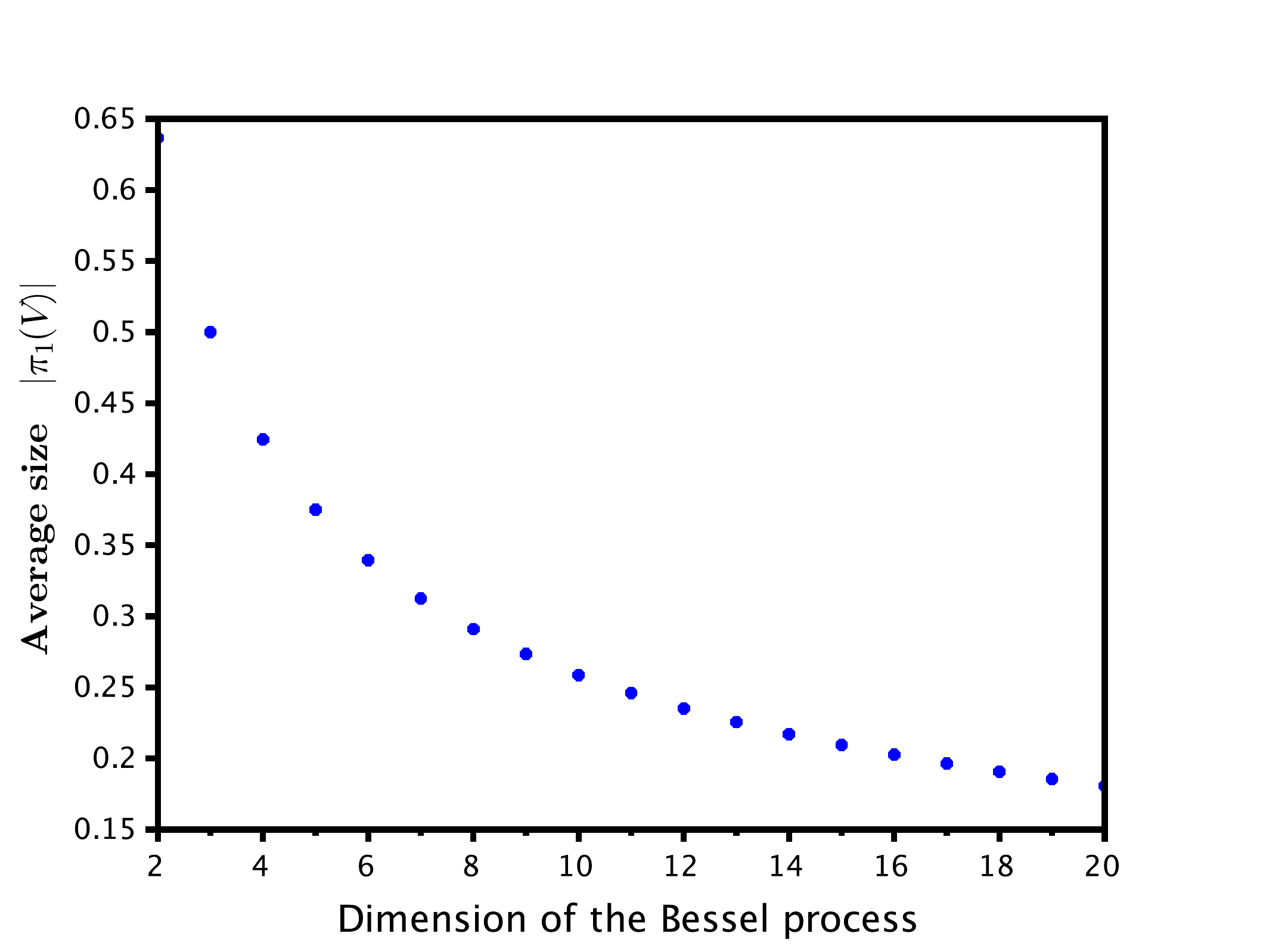}}
\end{minipage}\\[4pt]
 
We notice that this reduction is not too strong, for $\delta=20$ for instance the reduction corresponds to a division by $5$.

The efficiency of the approximation is deeply related to the number of spheroids used to cover the time interval $[0,T]$. Theorem \ref{thm:Bess1} (Central Limit Theorem) points out the asymptotic result as $\varepsilon$ tends to $0$ for Bessel processes with integer dimensions. Numerical experiments permit to obtain an histogram of the number of points for the generation of $10\,000$ skeletons, see Figure \ref{fig:my_label}.

\begin{figure}[h]
\centering
\includegraphics[width=8.5cm]{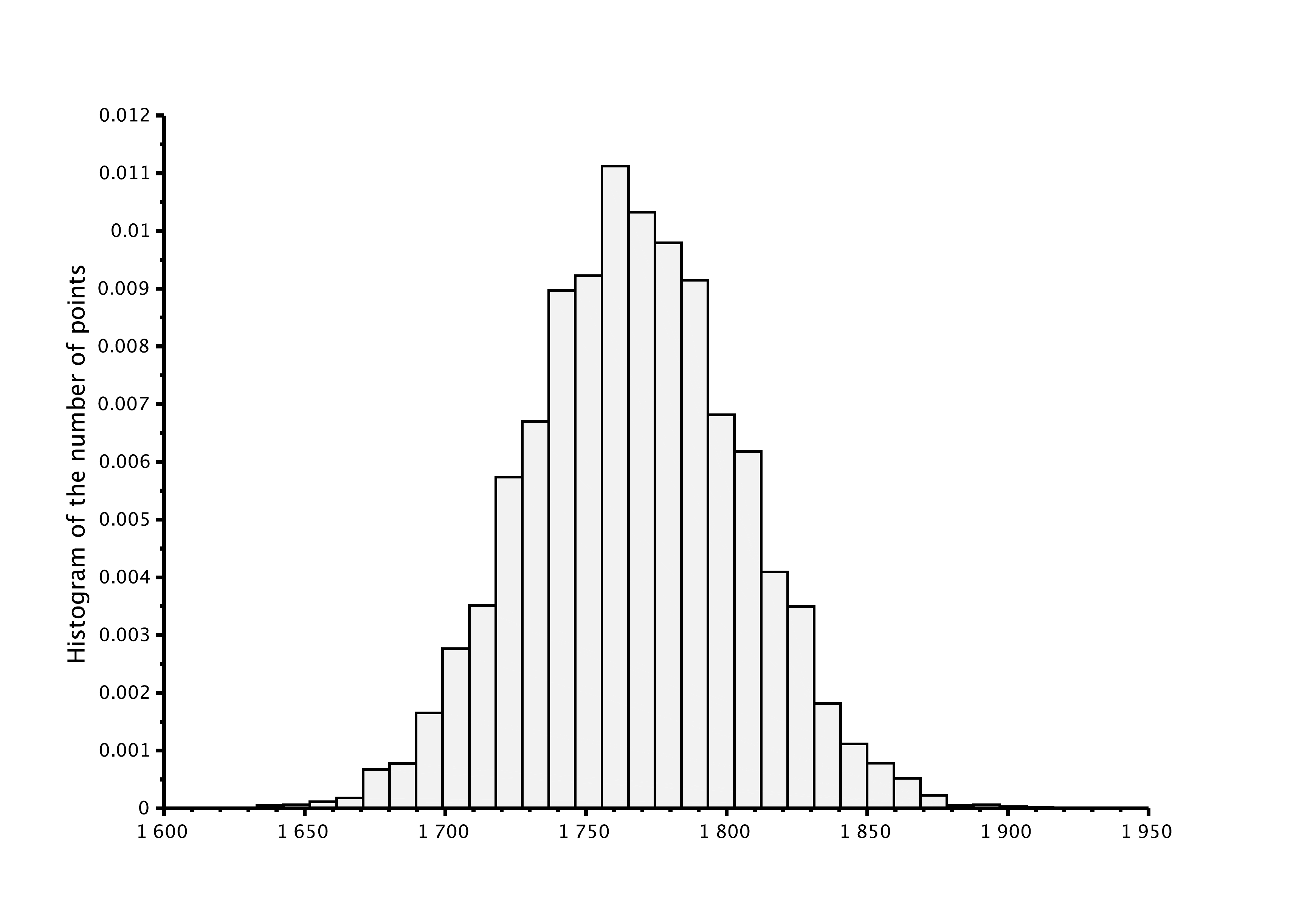}\hspace*{-0.5cm}\includegraphics[width=8.5cm]{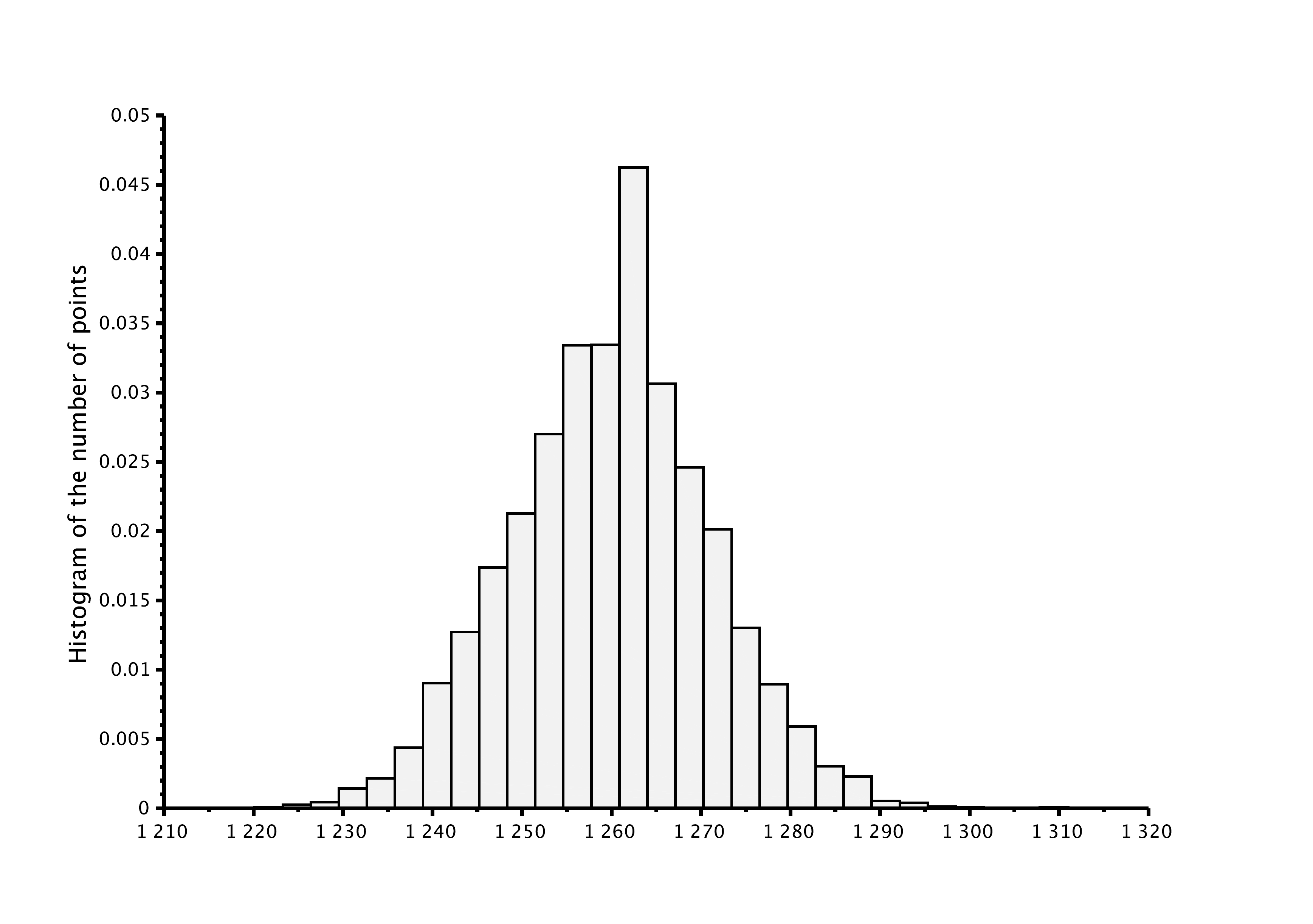}
\caption{\small Histogram of the number of points $N_T^\varepsilon$ for a sample of Bessel processes of dimension $2$ (left) or dimension $20$ (right) observed on the time interval $[0,3]$ with $y_0=0.5$, $\varepsilon=0.05$, sample size $10\,000$.}
\label{fig:my_label}
\end{figure}

A characteristic of the asymptotic behaviour is the mean number of spheroids necessary to cover some time interval $[0,T]$. We propose here to estimate it by using an empirical mean issued from a sample of $1\ 000$ trajectories. As already mentioned, we observe a dependence with respect to the Bessel dimension, the number of spheroids used by the algorithm increases as $\delta$ increases. Figure \ref{fig:my_labell} emphasizes that this dependence looks linear. Moreover the estimation of the average permits to illustrate the asymptotic linear dependence with respect to the parameter $1/\varepsilon^2$, here $\varepsilon$ stands for the accuracy of the strong approximation. 

\begin{figure}[H]
\centering
\includegraphics[width=8cm]{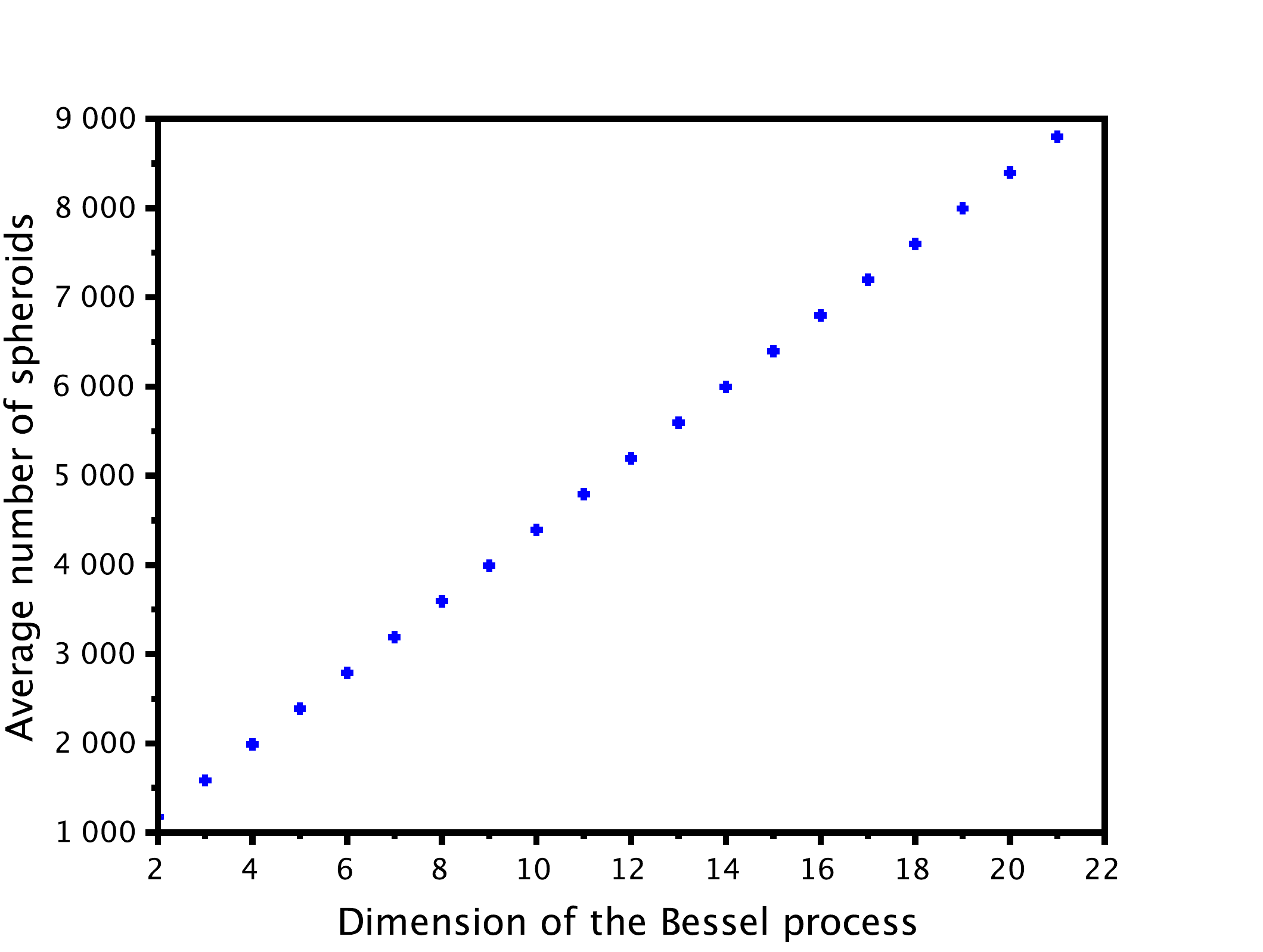}\hspace*{-0.5cm}\includegraphics[width=8.7cm]{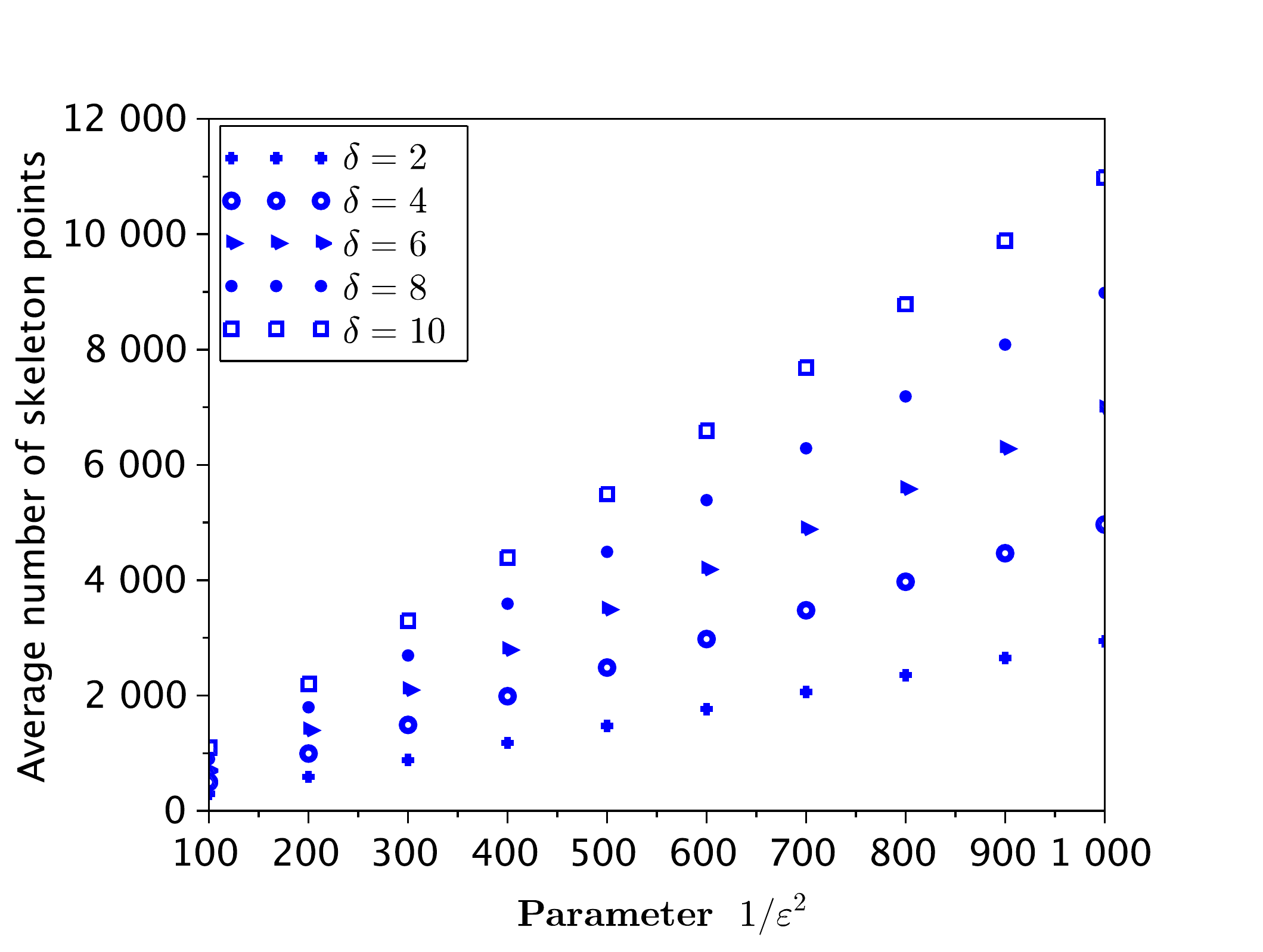}
\caption{\small Average number of skeleton points on the time interval $[0,1]$ versus the dimension of the Bessel process (left, with $\varepsilon=0.05$) or versus the parameter $1/\varepsilon^2$ (right) with $y_0=0.5$ and a sample size for the computation of the empirical mean $1\,000$.}
\label{fig:my_labell}
\end{figure}

In order to completely illustrate the strong approximation of the Bessel processes, let us consider numerical experiments for non integer dimensions. In this case,  Algorithm $({\rm BeS})_\delta^w$ permits to generate the Bessel skeletons. Of course, due to the decomposition related to Shiga-Watanabe's property, we need to observe both a sequence of spheroids for the Bessel process corresponding to the integer part of the dimension and a sequence of spheroids for the fractional part. That's why it is reasonable to see a large number of skeleton points in order to approximate the paths. For instance, for a Bessel process of dimension $\delta=2.2$, the average of this random number represented by the histogram of Figure \ref{fig:my_label:bis} (left) is about  $68\,130$ 
 while the average in the particular $d=2$ dimension (Figure \ref{fig:my_label} -- left) is approximately equal to $1767$. This sharp increase strongly depends on the value of the parameter $w$ which determines the size of the spheroids of both the integer and fractional part of the algorithm. The challenge is therefore to obtain a balanced repartition. The optimal choice of the parameters $(w_i, w_f)$, satisfying the identity $w_f+2\sqrt{w_i}=1$, is illustrated by different numerical experiments in Figure \ref{fig:my_label:bis} (right). We observe that this optimal choice depends on the Bessel dimension and can be compared to the heuristic choice suggested in Corollary \ref{cor:bessne} which is represented by a vertical line in the figure.

\begin{figure}[h]
\centering
\includegraphics[width=8.5cm]{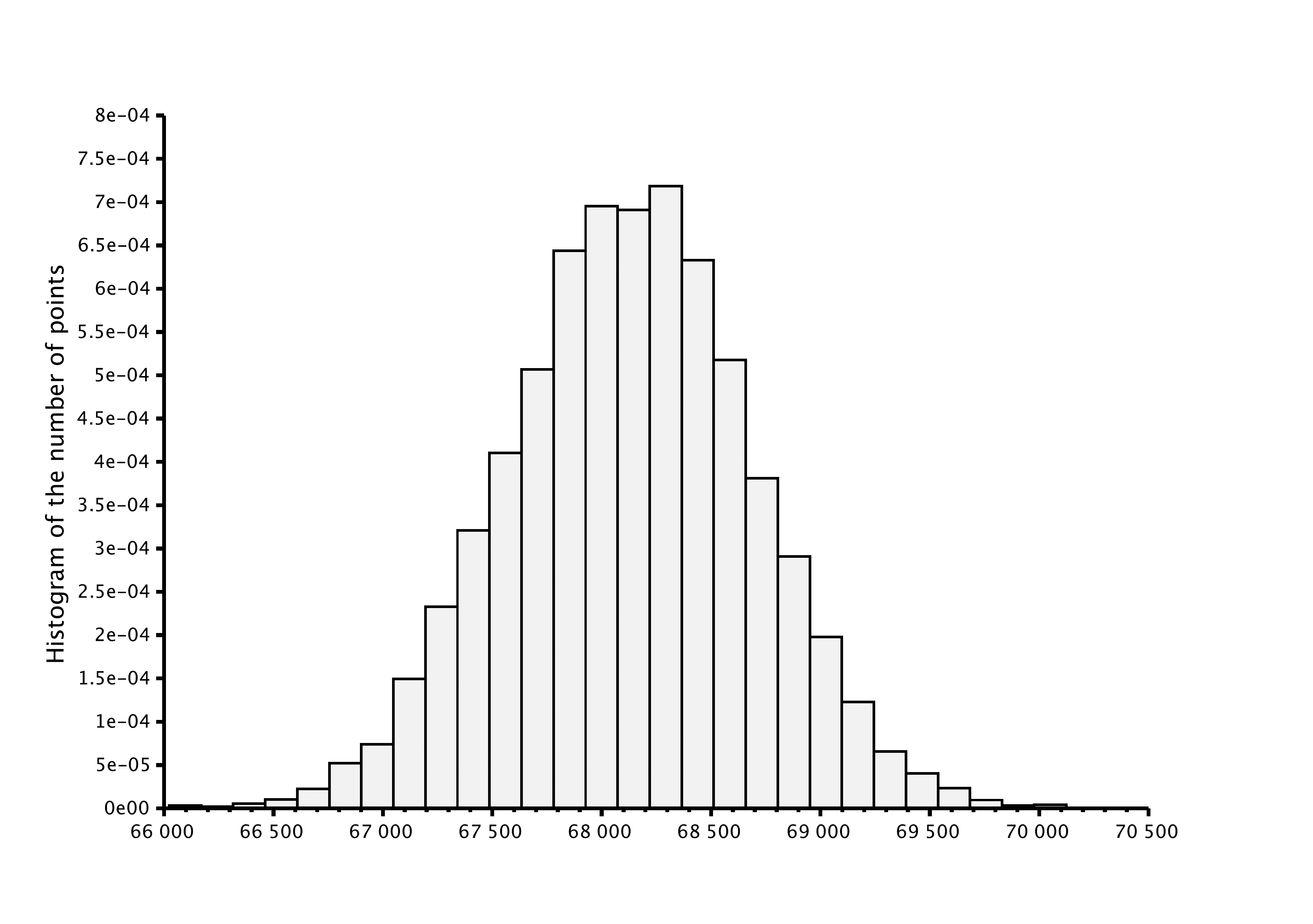}\hspace*{-0.5cm}\includegraphics[width=8.5cm]{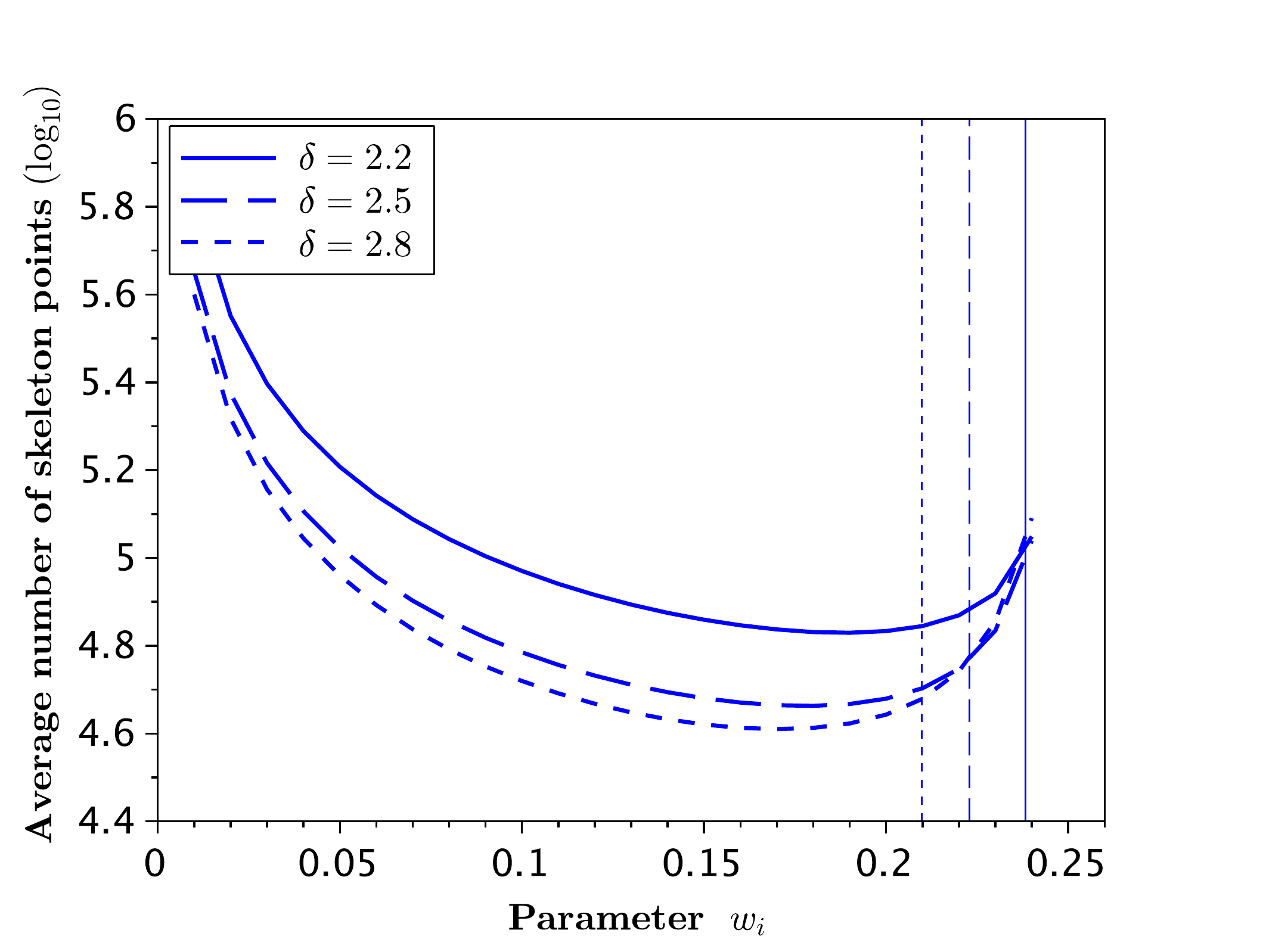}
\caption{\small Histogram of the number of points $N_T^\varepsilon$ for a sample of Bessel processes of dimension $2.2$  with $w_f=0.5$ and $w_i=0.0625$ (left) -- Average number of skeleton points (in $\log_{10}$-scale) versus the parameter value $w_i$ (right). Both figures concern paths observed on the time interval $[0,3]$ with $y_0=0.5$, $\varepsilon=0.05$. Sample sizes: $10\,000$ (left), $1\,000$ for the empirical mean (right).}
\label{fig:my_label:bis}
\end{figure}
\subsection{Related processes}
Several stochastic processes related to the Bessel one play an important role in the finance literature. Here the aim of the discussion is not to present a complete overview of financial models which could be concerned by our approximation procedure but rather to present few examples. Let us first recall the statement of  Definition \ref{def}:  $(y_t^\varepsilon)$ is an $\varepsilon$-strong approximation of the diffusion process $(X_t)$ on the fixed time interval $[0,T]$ if there exists $(x^\varepsilon_t)$ satisfying
\begin{equation*}
\sup_{t\in[0,T]}|X_t-x^\varepsilon_t|\le \varepsilon\quad\mbox{a.s.}
\end{equation*}
such that $(y^\varepsilon_t)$ and $(x^\varepsilon_t)$ are identically distributed. Consequently, as a by-product, any approximation of the Bessel path $(Z^{\delta,y}_t,\, t\le T)$ leads to an approximation of the path $(Y_t,\,t\le T_0)$ defined by
\begin{equation}\label{eq:rel}
Y_t:=f(t,Z^{\delta,y}_{\rho(t)}),
\end{equation}
with $f:\mathbb{R}_+^2\to \mathbb{R}$ a continuous function and $\rho:\mathbb{R}_+\to\mathbb{R}_+$ a strictly monotonous time change function. Of course the identity \eqref{eq:rel} implies a change of accuracy for the approximation and of course a change in the time interval under consideration. This adaptation is rather immediate and permits to handle with a large class of processes. In the family of financial term structure models, we can for instance focus our attention on the square-root process or CIR model (Cox-Ingersoll-Ross). This process appearing in the seminal paper of Cox et al. \cite{cox1985theory} is the object of many studies and is simply defined as the positive solution of 
\begin{equation}\label{eq:CIR}
dY_t=k(\theta-Y_t)\,dt+\sigma \sqrt{Y_t}\,dB_t,\quad Y_0=x,
\end{equation}
under the conditions $k\theta>0$ and $\sigma>0$. Using stochastic calculus permits to point out that the process $Y$ satisfies (not especially with respect to the same Brownian motion) \eqref{eq:rel} with
\[
f(t,x)=e^{-kt}x^2,\quad \rho(t)=\frac{\sigma^2}{4k}\,(e^{kt}-1),\quad \delta=\frac{4k\theta}{\sigma^2}\ \ \mbox{and}\ \  y=\sqrt{x}.
\]
Let us note that the coefficients of the diffusion \eqref{eq:CIR} are time-homogeneous. It is possible to extend this family of term structure models to inhomogeneous processes (see, for instance \cite{jeanblanc2009mathematical}) solution to
\[
dY_t=(a-\lambda(t)Y_t)\,dt+\sigma \sqrt{Y_t}\,dB_t,\quad Y_0=x,
\]
where $\lambda$ is a continuous function. We are still able to emphasize a relation like \eqref{eq:rel} with the following functions and parameters (see, for instance Theorem 6.3.5.1 in \cite{jeanblanc2009mathematical}):
\[
f(t,x)=\frac{\sigma^2}{4\rho'(t)}\,x^2,\quad \rho(t)=\frac{\sigma^2}{4}\,\int_0^t\exp\Big\{\int_0^s \lambda(u)\,du\Big\}\,ds,\quad \delta=\frac{4a}{\sigma^2}\ \ \mbox{and}\ \  y=\sqrt{x}.
\]
Both the homogeneous and the inhomogeneous CIR models are related to the squared Bessel process through a time dependent linear transformation. Modelling the volatility in finance actually requires to handle with other process: the CEV model (Constant Elasticity of Variance) which satisfies:
\[
dY_t=Y_t(\mu\,dt+\sigma Y_t^{\beta}\,dB_t),\quad t\ge 0,\quad Y_0=x.
\]
Under particular conditions, the process $(Y_t)_{t\ge 0}$ satisfies \eqref{eq:rel} with $f(t,x)=e^{\mu t}x^{\alpha}$, $\alpha$ depending on $\beta$ and being different from the square (see for instance \cite{jeanblanc2009mathematical}). For option pricing in finance, it is therefore of prime interest to simulate precisely trajectories of underlying assets which follow CIR or CEV models. It permits to estimate the prices of derivatives like European options but also paths dependent options like Asian or barrier options.

As already seen, families of stochastic models are directly related to the Bessel process through the identity \eqref{eq:rel}. If the function $f$ is globally Lipschitz continuous with respect to the space variable then the Bessel $\varepsilon$-strong approximation $(y_t^\varepsilon)_{t\ge 0}$ allows to generate a $\varepsilon'$-approximation of $(Y_t)_{t\ge 0}$ which is given by $(f(t,y_{\rho(t)}^\varepsilon))_{t\ge 0}$, the parameters $\varepsilon$ and $\varepsilon'$ being related through the Lipschitz constant and the time interval under consideration. 

If the transformation $f$ is not uniformly Lipschitz with respect to the space variable (CIR and CEV models, for instance), then the Bessel $\varepsilon$-strong approximation permits to obtain a lower-bound and an upper-bound of any path $(Y_t)_{t\ge 0}$ depending on $\varepsilon$. These bounds imply a precise estimation of path-dependent characteristics and play therefore a crucial role for applications. Let us consider the following example: a CIR model observed on the time interval $[0,2]$ with the parameters: $k=2$, $\theta=1/3$, $\sigma =1$ and the starting value $x=1$. It is therefore expressed by $Y_t=f(t,Z_{\rho(t)}^{\delta,y})$ for all $t\in[0,2]$. Introducing the $\varepsilon$-strong approximation of the Bessel process $(y^\varepsilon_t)_{t\ge 0}$, based on the Bessel skeleton  $({\rm BeS})_\delta$ or $({\rm BeS})_\delta^w$ that is $((u_n^\varepsilon,s_n^\varepsilon)_{n\ge 1},(y_n^\varepsilon)_{n\ge 0})$, we obtain the almost surely bounds:
\[
f(t,y_{\rho(t)}^\varepsilon-\varepsilon)\le Y_t \le f(t, y_{\rho(t)}^\varepsilon+\varepsilon),\quad \forall t\in[0,2],
\]
since the function $x\mapsto f(t,x)$ is increasing. In Figure \ref{fig:my_label:CIR} (right), one generation of the upper and lower bounds is represented for any $t\in \{\rho^{-1}(s_n)\}_{n\ge 1}\cap[0,2]$.  The accuracy of the approximation is not uniform since it depends on the value of the process and on the time variable. More precisely, we propose to define the precision variable $P_\varepsilon$ by
\begin{equation}
\label{eq:def:precis}
P_\varepsilon:=\sup_{t\in[0,2]}\Big|f(t, y_{\rho(t)}^\varepsilon+\varepsilon)-f(t,y_{\rho(t)}^\varepsilon-\varepsilon)\Big|.
\end{equation}
Using the explicit expression of the function $f$ associated with the CIR model, we obtain an explicit expression of the accuracy depending on the Bessel skeleton:
\[
P_\varepsilon=4\,\varepsilon\sup \Big\{y_{n}\, e^{-2\rho^{-1}(s_n)}\ \mbox{s.t.}\ s_n\le \rho^{-1}(2)\Big\}.
\]
The probability distribution of the ration $P_\varepsilon/\varepsilon$ is represented in Figure \ref{fig:my_label:CIR} (left): we observe that the accuracy is close to four times the initial condition of the Bessel process. 

\begin{figure}[h]
\centering
\includegraphics[width=8.5cm]{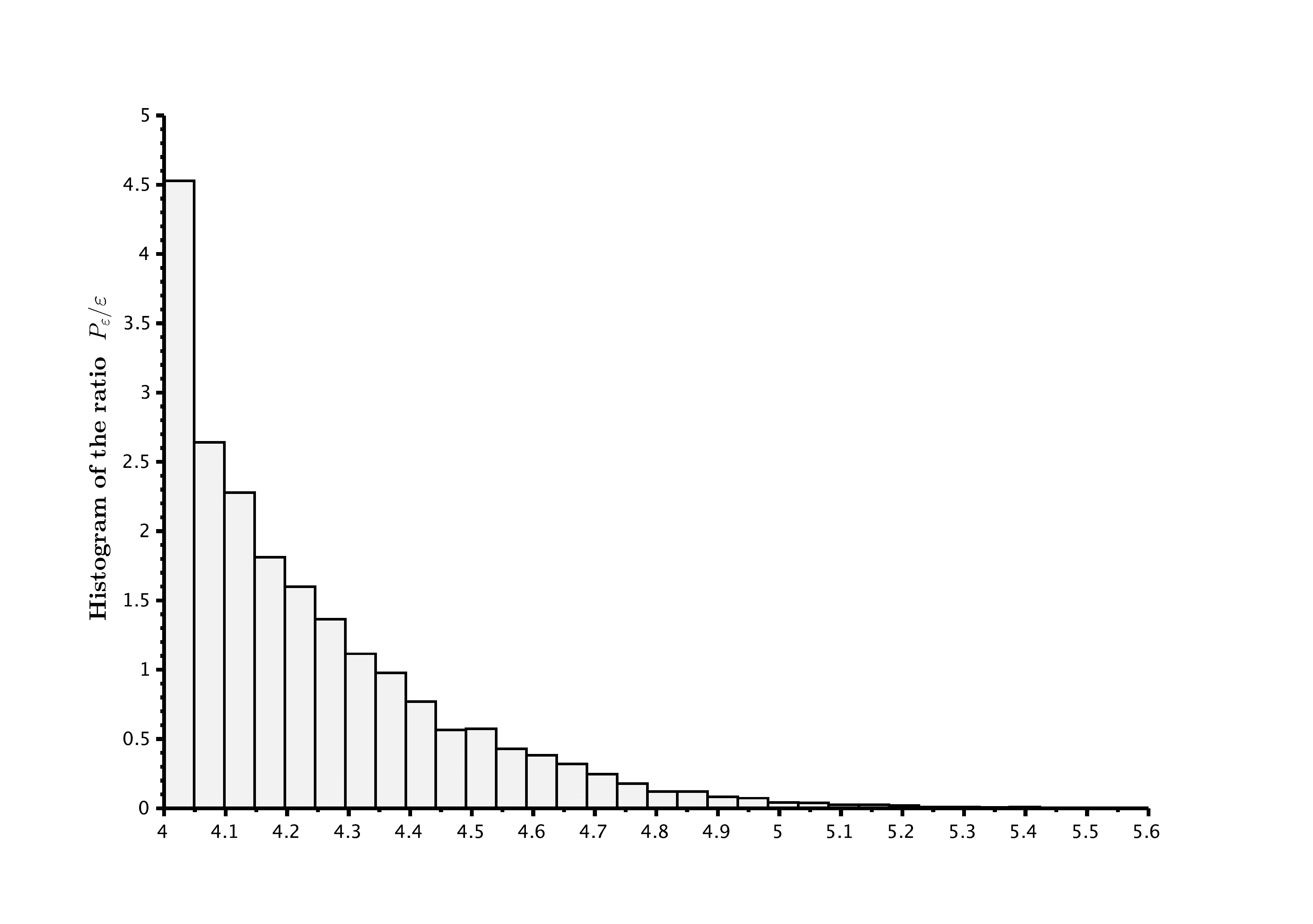}\hspace*{-2.4cm}\includegraphics[width=11cm]{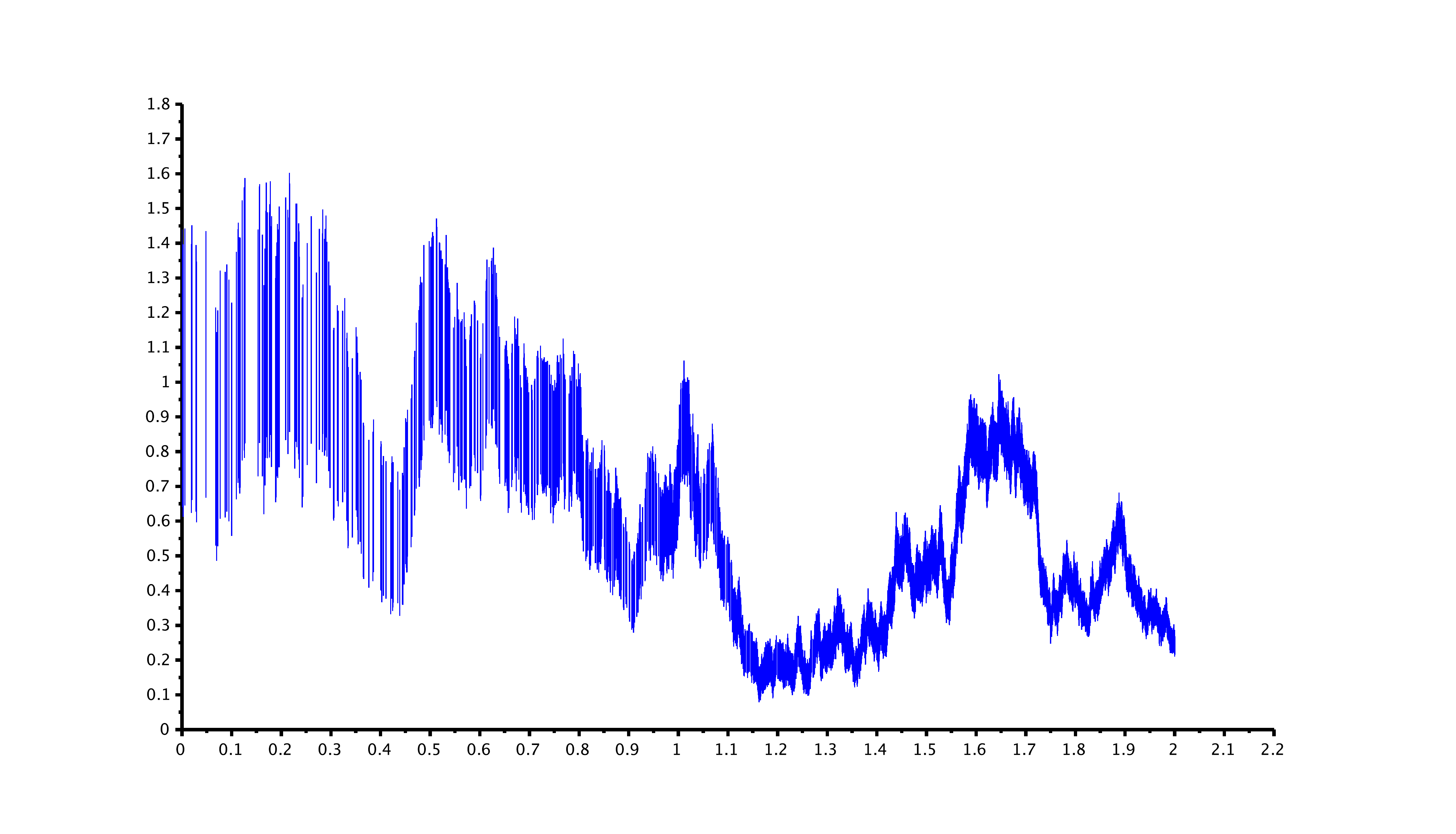}
\caption{\small Histogram of the precision variable $P_\varepsilon/\varepsilon$ (left) and a generation of bounding points (joined by vertical lines). Both figures concern paths of the CIR model observed on the time interval $[0,2]$ with $x=1$, $k=2$, $\theta=1/3$, $\sigma=1$.
Histogram (left): sample size $10\,000$ and $\varepsilon=0.05$. Bounding points (right): $\varepsilon=0.2$.
}
\label{fig:my_label:CIR}
\end{figure}

Of course the difference between the lower and upper paths is not uniformly bounded. This accuracy is nevertheless sufficient in many applications but if the challenge is to reach a uniform bound, then we suggest another approach. The key is to let the size of the spheroids used in the Bessel approximation depend on the space variable: the size is no more fixed once for ever and equal to $\varepsilon$. Such an approach was presented in detail in \cite{deaconu2020strong} for processes defined by $Y_t:=f(t,B_{\rho(t)})$, transformations of the time-changed Brownian motion and can be adapted to the Bessel case.
\bibliographystyle{alpha}
\bibliography{biblio}
\end{document}